\documentclass[11pt]{article}

\usepackage{amsbsy,amsfonts,amsmath,amssymb,mathrsfs}
\usepackage[T1]{fontenc}
\usepackage{tikz}
\usetikzlibrary{cd,arrows,matrix,positioning}
\usepackage{hyperref}
\usepackage{cleveref}
\usepackage[toc,page,title]{appendix}

\makeatletter
\let\oldfootnote\footnote
\def\footnote{\@ifstar\footnote@star\footnote@nostar}
\def\footnote@star#1{{\let\thefootnote\relax\footnotetext{#1}}}
\def\footnote@nostar{\oldfootnote}
\makeatother

\newcommand{\spref}[1]{\cite[\href{http://stacks.math.columbia.edu/tag/#1}{Tag~#1}]{SP22}}

\def\itemn#1{\item[\hspace{0.6mm} {\rm (#1)}]}
\def\itemm#1{\item[\indent {\rm (#1)}]}

\renewcommand{\ge}{\geqslant}
\renewcommand{\le}{\leqslant}
\newcommand{\sHom}{\sH\!om}
\newcommand{\sThick}{\sT\!hick}
\newcommand{\sExal}{\sE\!xal}

\def\toto{\rightrightarrows}
\def\too{\longrightarrow}
\def\tooo{\relbar\joinrel\longrightarrow}
\def\into{\hookrightarrow}
\newcommand*{\intoo}{\ensuremath{\lhook\joinrel\relbar\joinrel\rightarrow}}

\def\isomto{\xrightarrow{\,\smash{\raisebox{-0.5ex}{\ensuremath{\scriptstyle\sim}}}\,}}

\renewcommand{\hat}{\widehat}
\newcommand{\eps}{\varepsilon}

\newcommand{\smallvee}{{\scriptscriptstyle\vee}}

\makeatletter
\newcommand*{\defeq}{\mathrel{\rlap{%
                     \raisebox{0.3ex}{$\m@th\cdot$}}%
                     \raisebox{-0.3ex}{$\m@th\cdot$}}%
                     =}
\makeatother

\usetikzlibrary{decorations.markings}

\makeatletter
\tikzcdset{
  open/.code     = {\tikzcdset{hook, circled};},
  closed/.code   = {\tikzcdset{hook, slashed};},
  open'/.code    = {\tikzcdset{hook', circled};},
  closed'/.code  = {\tikzcdset{hook', slashed};},
  circled/.code  = {\tikzcdset{markwith = {\draw (0,0) circle (.375ex);}};},
  slashed/.code  = {\tikzcdset{markwith = {\draw[-] (-.4ex,-.4ex) -- (.4ex,.4ex);}};},
  markwith/.code ={
    \pgfutil@ifundefined%
    {tikz@library@decorations.markings@loaded}%
    {\pgfutil@packageerror{tikz-cd}{You need to say %
      \string\usetikzlibrary{decorations.markings} to use arrows with markings}{}}{}%
    \pgfkeysalso{/tikz/postaction = {
      /tikz/decorate,
      /tikz/decoration={markings, mark = at position 0.5 with {#1}}}
    }
  },
}
\makeatother

\newtheorem{counter}[subsubsection]{$\!\!$}
\newtheorem{subcounter}[subsection]{$\!\!$}
\newcounter{intro}
\renewcommand{\theintro}{\arabic{intro}}
\newenvironment{definition}{\begin{counter} \rm {\bf Definition.}}{\end{counter}}
\newenvironment{definition*}{\begin{subcounter} \rm {\bf Definition.}}{\end{subcounter}}
\newenvironment{theorem}{\begin{counter} {\bf Theorem.}}{\end{counter}}
\newenvironment{theorem-quote}[1]
{\begin{counter} {\bf Theorem (#1).}}{\end{counter}}
\newenvironment{theorem*}{\begin{subcounter} {\bf Theorem.}}{\end{subcounter}}
\newenvironment{theorem-intro}
{\refstepcounter{intro}\bigskip\noindent {\bf  Theorem~\theintro.}\em\!\!\!}{\rm\bigskip}

\newenvironment{proposition}{\begin{counter} {\bf Proposition.}}{\end{counter}}
\newenvironment{lemma}{\begin{counter} {\bf Lemma.}}{\end{counter}}
\newenvironment{lemma*}{\begin{subcounter} {\bf Lemma.}}{\end{subcounter}}

\newenvironment{corollary}{\begin{counter} {\bf Corollary.}}{\end{counter}}
\newenvironment{corollary*}{\begin{subcounter} {\bf Corollary.}}{\end{subcounter}}
\newenvironment{remark}{\begin{counter} \rm {\bf Remark.}}{\end{counter}}
\newenvironment{remark*}{\begin{subcounter} \rm {\bf Remark.}}{\end{subcounter}}

\newenvironment{example}{\begin{counter} \rm {\bf Example.}}{\end{counter}}
\newenvironment{examples}{\begin{counter} \rm {\bf Examples.}}{\end{counter}}
\newenvironment{notitle}[1]{\begin{counter} {\bf #1.}\rm}{\end{counter}}
\newenvironment{notitle*}[1]{\begin{subcounter} {\bf #1.}\rm}{\end{subcounter}}
\newenvironment{proof}{{\flushleft \bf Proof~:}}{\hfill $\square$ \vspace{5mm}}

\DeclareMathOperator{\Spec}{Spec}
\DeclareMathOperator{\Frob}{F}

\DeclareMathOperator{\et}{\acute{e}t}
\DeclareMathOperator{\fppf}{fppf}

\DeclareMathOperator{\colim}{colim}
\DeclareMathOperator{\Hilb}{Hilb}
\DeclareMathOperator{\Hom}{Hom}
\DeclareMathOperator{\Ext}{Ext}
\DeclareMathOperator{\Aut}{Aut}
\DeclareMathOperator{\Isom}{Isom}
\DeclareMathOperator{\Mono}{Mono}

\DeclareMathOperator{\id}{id}
\DeclareMathOperator{\pr}{pr}

\DeclareMathOperator{\Ass}{Ass}

\DeclareMathOperator{\GL}{GL}
\DeclareMathOperator{\SO}{SO}

\DeclareMathOperator{\Res}{Res}
\DeclareMathOperator{\Resgr}{GRes}
\DeclareMathOperator{\Sch}{\bf Sch}

\DeclareMathOperator{\Sp}{Sp}
\DeclareMathOperator{\length}{length}
\DeclareMathOperator{\can}{can}
\DeclareMathOperator{\mult}{mult}

\DeclareMathOperator{\red}{red}
\DeclareMathOperator{\type}{\mathbf{t}}

\DeclareMathOperator{\Sub}{Sub}
\DeclareMathOperator{\Norm}{Norm}

\DeclareMathOperator{\Liset}{Lis-\acute{E}t}
\DeclareMathOperator{\triv}{triv}
\DeclareMathOperator{\str}{str}
\DeclareMathOperator{\Ind}{Ind}
\DeclareMathOperator{\Thick}{Thick}
\DeclareMathOperator{\Exal}{Exal}

\DeclareMathOperator{\QCoh}{QCoh}
\DeclareMathOperator{\Eq}{Eq}

\DeclareMathOperator{\val}{val}

\DeclareMathOperator{\inn}{inn}
\DeclareMathOperator{\St}{St}
\DeclareMathOperator{\modulo}{\ mod }

\DeclareMathAlphabet{\mathcalligra}{T1}{calligra}{m}{n}

 \def\cF{{\cal F}}  \def\cH{{\cal H}}

\def\cO{{\cal O}}

\def\cZ{{\cal Z}}

\newcommand\CC{\mathbb{C}}

\newcommand\GG{\mathbb{G}}
\newcommand\HH{\mathbb{H}}

\newcommand\LL{\mathbb{L}}

\newcommand\PP{\mathbb{P}}
\newcommand\QQ{\mathbb{Q}}

\newcommand\ZZ{\mathbb{Z}}

  \def\sE{\mathscr{E}}
\def\sF{\mathscr{F}} \def\sG{\mathscr{G}} \def\sH{\mathscr{H}}
  
 \def\sP{\mathscr{P}} \def\sQ{\mathscr{Q}}
\def\sR{\mathscr{R}}  \def\sT{\mathscr{T}}

\def\sX{\mathscr{X}}  

\begin{document}

\quad\bigskip
\bigskip

\begin{center}
{\bf \LARGE Algebraicity and smoothness of fixed point
stacks}

\bigskip

Matthieu Romagny
\end{center}

\bigskip

\hfill {\em To the memory of Bas Edixhoven}

\bigskip
\bigskip

\begin{center}
\begin{minipage}{14cm}
{\small {\bf Abstract.} We study algebraicity and smoothness
of fixed point stacks for flat group schemes which have a
finite composition series whose factors are either reductive
or proper, flat, finitely presented, acting on algebraic stacks
with affine, finitely presented diagonal. For this, we extend
some theorems of \cite{SGA3.2} on functors of homomorphisms
$\Hom(G,H)$ and functors of reductive subgroups $\Sub(H)$
for an affine, possibly non-flat group scheme $H$.}
\end{minipage}
\end{center}

\footnote*{
\hspace{-7mm} 2020 Mathematics Subject Classification:
Primary: 14A20, Secondary: 14L15, 14D23}

\footnote*{
\hspace{-7mm} Keywords: fixed points, algebraic stack,
$\Hom$ functor, reductive group scheme,
linearly reductive group scheme}

\footnote*{
\hspace{-7mm} Date: \today}

\section{Introduction}

\begin{notitle*}{Context and motivation}
In various situations of algebraic geometry, one needs to
consider the fixed points of a flat group scheme acting
on an algebraic stack. Currently, probably the biggest provider
of such examples is the enumerative industry: Gromov-Witten
and Donaldson-Thomas theories provide a wealth of apparitions
of fixed points in localization formulas for virtual
classes in equivariant cohomology. We refer to Joyce~\cite{Jo21}
for a recent account. Accordingly, fixed point stacks pervade
research articles in the last two decades; with no attempt at
exhaustivity, let us mention the works \cite{CLCT09}, \cite{Di12},
\cite{We11}, \cite{Sk13}, \cite{KL13}, \cite{GJK17},
\cite{OS19}, \cite{MR19}, \cite{KS20}, \cite{LS20},
\cite{BTN21}, \cite{CJW21}. The paper~\cite{Ro05} settles
the question of algebraicity of fixed point stacks only
in the case of actions of {\em proper}
groups, which limitates the scope of applications (and is
typically not sufficient in most works cited above).
It is the purpose of the present article to
extend the results of {\em loc. cit.}, providing algebraicity
and smoothness statements for fixed point stacks in greater
generality. A related issue is that of representability
of functors of group homomorphisms $\Hom(G,H)$ and functors
of subgroups $\Sub(H)$. In \cite{SGA3.2}, Exp.~X1, \S~4 such
representability is proved in the case where the target group
scheme $H$ is {\em smooth}. Unfortunately, for the application
to the fixed points of a group $G$ acting on an algebraic
stack $\sX$, it is the inertia $I_{\sX}\to\sX$ which plays the
role of~$H$, and this is almost never flat. We also explore
these issues, working with a possibly non-flat group $H$.
We answer questions raised in \cite{SGA3.2} by both relaxing
the assumptions and strengthening the results.
\end{notitle*}

\begin{notitle*}{Main results}
Throughout the paper, we denote by $S$ the base scheme.
To get to the heart of the matter we need to recall some
terminology. Following \cite{SGA3.2}, Exp.~XIX, 2.7 and
\cite{AOV08}, 2.2 or \cite{Alp13}, 12.1,
we say that a group scheme $G\to S$ is:
\begin{trivlist}
\itemn{i} {\em reductive} if it is affine and smooth
with connected, reductive geometric fibres;
\itemn{ii} {\em linearly reductive} if it is flat,
separated, of finite presentation, and the functor
$\QCoh^G(S)\to\QCoh(S)$, $\cF\mapsto \cF^G$ is exact.
This includes: group schemes of
multiplicative type, finite locally free group schemes of
order invertible on $S$, abelian schemes, reductive group
schemes if $S$ is a $\QQ$-scheme, and all extensions of such
group schemes (we refer to~\ref{defi:lin_red_gp} and the
comments after it).
\end{trivlist}

Here is our result on fixed point stacks;
see \ref{th-fixed-pt-stack} and
\ref{thm:smoothness-fixed-points}.


\begin{theorem} \label{th-1}
Let $\sX$ be an $S$-algebraic stack with affine,
finitely presented diagonal. Let $G$ be a flat, finitely
presented $S$-group algebraic space acting on $\sX$.
\begin{trivlist}
\itemn{1} Assume that $G$ has a finite composition series
whose factors are either reductive or proper, flat, finitely
presented. Then the fixed point stack $\sX^G$ is algebraic,
and the morphism $\sX^G\to \sX$ is representable by algebraic
spaces, separated and locally of finite presentation.
If $G$ is reductive, this morphism is even representable
by schemes.
\itemn{2} Assume that $\sX$ is smooth and $G$ is
linearly reductive. Then $\sX^G$ is smooth.
\end{trivlist}
\end{theorem}

The proof necessitates results on functors of group homomorphisms.
Before stating them,
we bring to the reader's attention the subtle question
of the (non-)affineness of $\Hom(G,H)$, when it is
representable. Quoting \cite{SGA3.2}, Exp.~XI,
Rem.~4.6 we know that if $G$ is of multiplicative type and
$H$ is a closed subgroup of some $\GL_n$, then $\Hom(G,H)$
is a disjoint sum of affine schemes, but `on se gardera de
croire cependant que les pr\'esch\'emas qui repr\'esentent
[ces] foncteurs sont toujours des sommes d'une famille de
sch\'emas aﬃnes sur $S$' (`the reader should refrain
from thinking that the schemes representing these functors
are always a sum of a family of affine $S$-schemes').
In recent work of Brion, the same problem is encountered
and certain conditions (FT) and (AFT) are introduced in order
to best describe this phenomenon;
see \cite{Bri21}, \S~4.2. The statement in item~(3) below
is our contribution to this question, in the present
generality.

The following result is found in
\ref{theorem:Hom_reductive_affine},
\ref{theo:Hom_from_proper_reductive},
\ref{cor:smoothness-Hom-functor}.

\begin{theorem} \label{th-2}
Let $G$ be an $S$-group space that has a finite
composition series whose factors are either reductive or
proper, flat, finitely presented.
Let $H$ be an affine, finitely presented $S$-group scheme.
\begin{trivlist}
\itemn{1} The functor $\Hom(G,H)$ is representable by an
$S$-algebraic space separated and locally of finite presentation.
\itemn{2} The subfunctor of monomorphisms $\Mono(G,H)$ is
representable by an open subspace of $\Hom(G,H)$. Moreover,
all monomorphisms $G\to H$ are closed immersions.
\itemn{3} If $G$ is reductive, $\Hom(G,H)$ is
representable by a scheme with the following property: each
subscheme (resp. closed subscheme) which is quasi-compact
over $S$, is quasi-affine (resp. affine) over $S$.
\itemn{4} If $G$ is linearly reductive and $H$ is flat,
the algebraic stack $\sHom(BG,BH)$ is smooth. In particular,
\begin{trivlist}
\itemm{i} $\Hom(G,H)\to S$ is flat and locally complete
intersection,
\itemm{ii} $\Hom(G,H)\to S$ is smooth if moreover $H\to S$
is smooth.
\end{trivlist}
\end{trivlist}
\end{theorem}

The assumptions on $G$ and $H$ are close to optimal.
Indeed, it is classical that the presence of unipotent factors
in $G$ is an obstacle to representability; taking $G=\GG_a$
and $H=\GG_m$ for simplicity, this is due to the existence of
exponentials that do not algebrize, violating effectivity
(the axiom called $(F_3)$ in the text). Also and assumption
``affine'' or at least ``quasi-affine'' on $H$ is necessary, as
shown by the example of an elliptic curve with multiplicative
reduction (\cite{SGA3.2}, Expos\'e~IX, Rem.~7.4).

Our third and last main result is about functors of
reductive subgroups of affine, possibly non-flat group schemes.
We refer to \ref{theorem:subgroups_reductive_affine}
and \ref{coro-sub-mult-smooth}.

\begin{theorem} \label{th-3}
Let $H$ be an affine, finitely presented $S$-group scheme.
\begin{trivlist}
\itemn{1} The functor $\Sub_{\red}(H)$ of reductive
subgroups of $H$ is representable by an algebraic space
separated and locally of finite presentation which is
a disjoint sum indexed by the types of reductive groups:
\[
\Sub_{\red}(H)=\underset{\type}{\textstyle\coprod}
\,\Sub_{\type}(H).
\]
\itemn{2} The summand of subgroups {\em of multiplicative type}
\[
\Sub_{\mult}(H)=\underset{\type=[(M,M^*,\varnothing,\varnothing)]}{\textstyle\coprod}
\,\Sub_{\type}(H)
\]
is representable by a {\em scheme} with the following property:
each subscheme (resp. closed subscheme) which is quasi-compact
over $S$, is quasi-affine (resp. affine) over $S$.
\itemn{3} Assume moreover that $H\to S$ is flat. Then
$\Sub_{\mult}(H)\to S$ is flat and locally complete
intersection, and smooth if $H\to S$ is. If $S$ is of
characteristic~0 then $\Sub_{\red}(H)\to S$ is smooth.
\end{trivlist}
\end{theorem}

\end{notitle*}

\begin{notitle*}{Main ideas of the proofs}
The proofs of the results \ref{th-2}(1)--(3) and \ref{th-3}(1)--(2)
are intertwined in a single line of reasoning. By d\'evissage
for $G$ we treat separately the reductive case and the proper flat
case. The key case is when $G$ is the multiplicative group $\GG_m$
(the case of proper flat
groups is standard using representability of the Hilbert scheme;
and one passes from general reductive groups to a maximal torus
using a result of \cite{SGA3.2} and then to $\GG_m$). For a
suitably chosen prime $\ell$, the density theorem shows that
$F:=\Hom(\GG_m,H)$ is a subfunctor of the affine scheme
$L:=\lim_n\Hom(\mu_{\ell^n},H)$. To complete the proof, one verifies
the eight axioms of Grothendieck's theorem on unramified functors for
the monomorphism $F\to L$. The axiom which is hardest to check is the
final one, and this is proved using a statement of ``descent along
schematically dominant morphisms'' whose proof is inspired from the
proof of algebraicity of formal homomorphisms.
From these results one deduces~\ref{th-1}(1).

As far as algebraicity statements are concerned, the results for
functors of homomorphisms of group schemes imply results for fixed
point stacks. For the statements of smoothness, things go in the
other direction.
That is, we first prove \ref{th-1}(2) by verifying the infinitesimal
criterion of smoothness for $\sX^G\to S$. For this we use the vanishing
of cohomology of linearly reductive group schemes to show that a
certain stack of liftings, which is a torsor under a certain tangent
stack, is trivial. Then we deduce the flatness
and smoothness statements~\ref{th-2}(4).
\end{notitle*}

\begin{notitle*}{Comments on related work}
The first general results on fixed points and homomorphism
functors are of course due to the work of the precursors
of \cite{SGA3.2}, \cite{SGA3.3}, \cite{Ra70}. These provide
the foundation for the results presented here.

Fixed point stacks are defined and studied in some generality
for actions of proper groups in \cite{Ro05}, of which the
results of the present text can be seen as a natural
continuation. In Subsection~\ref{failure} we take the
opportunity to correct a claim made in
\cite{Ro05}, Rem.~2.4 which turns out to be partially false.
Namely, say that $G,N$ are flat, finitely presented
group schemes with $N$ a normal subgroup of $G$. Assume
that~$G$ acts on a stack $\sX$ (we may take $\sX$ algebraic
and assume that the fixed point and quotient stacks below
are algebraic). Then, we provide an example where
$(\sX^N)^{G/N}$ and $\sX^G$ are {\em not} isomorphic
(in fact it is not clear how to let $G/N$ act on $\sX^N$
and we discuss this issue). On the other hand we prove that
there {\em is} always an isomorphism
of stacks $(\sX/N)/(G/N)\isomto \sX/G$.

In the paper~\cite{AHR20}, Alper, Hall and Rydh show that
when $\sX$ is a Deligne-Mumford stack locally of finite
type over a field with an action of $G=\GG_m$, then
$\sX^G\to\sX$ is a closed immersion. This can be easily
extended to the case where $\sX$ is a Deligne-Mumford stack,
$G$ is smooth with connected fibres, and the base scheme~$S$
is arbitrary. The Deligne-Mumford assumption is essential;
in Example~\ref{example-B-alpha-p} we show that~$\GG_m$ acts
on the classifying stack of $\alpha_p$-torsors (over a base
scheme of characteristic $p>0$) in such a way that
$\sX^{\GG_m}\to\sX$ is not a monomorphism. On the other
hand, if $\sX$ has finite inertia and the
group scheme~$G$ is smooth, the valuative criterion for
properness holds for $\sX^G\to\sX$.
In forthcoming work Aranha, Khan, Latyntsev, Park and Ravi
use this fact to prove a (virtual) Atiyah-Bott formula
under certain hypotheses.

In \cite{SGA3.2}, another feature of the scheme $M:=\Hom(G,H)$
is studied. Namely, for $h\in H$ let $\inn(h):H\to H$ be the
inner automorphism $k\mapsto hkh^{-1}$.
In \cite{SGA3.2}, Exp.~XI, \S~5 it is shown that if $G,H$
are finitely presented with $G$ of multiplicative type and
$H$ affine and smooth, then the morphism
\[
H\times M\too M\times M, \quad (h,v) \longmapsto (\inn(h)\circ v,v)
\]
is smooth. In the case of a base field, this is extended
by Brion~\cite{Bri21} to the situation where $G$ is linearly
reductive and $H$ is locally of finite type (but not
necessarily affine). Following the arguments of \cite{SGA3.2},
Exp.~IX 3.6 and Exp.~XI 2.3, with our running assumptions
`affine of finite presentation' on~$H$, it should
be possible to extend this further to the case where $S$
is arbitrary.

Recent work of Bhatt, Halpern-Leistner, Preygel (see
\cite{Bh16} Lemma~2.5, \cite{BHL17}, Section~2, \cite{HLP19},
Theorem~5.1.1) seems to indicate that it should be possible
to extend our results to algebraic stacks with
{\em quasi-affine} diagonal, and to target group schemes $H$
that are {\em quasi-affine}. We did not explore this possibility.

Moving away from linearly reductive group schemes, general
results on the smoothness of fixed points seem difficult
to obtain. Recent work of Hamilton \cite{Ha21} provides
an interesting attempt in this direction.

Finally we point out that the present text encompasses the
results of the preprint \cite{Ro21} which it supersedes.
\end{notitle*}

\begin{notitle*}{Organization of the paper}
The table of contents after the acknowledgements describes
the plan of the article.
\end{notitle*}

\begin{notitle*}{Acknowledgements}
I wish to express warm thanks to Arkadij Bojko, who provided
the initial stimulus with the question of algebraicity of
fixed point stacks. I thank Michel Brion for conversations
on the topic of this article, and for bringing the
paper \cite{Bri21} to my attention. For various conversations
and feedback, I also thank
Dhyan Aranha, Alice Bouillet, Pierre-Emmanuel Chaput,
Philippe Gille, Luc Illusie, Marion Jeannin, Bernard Le Stum,
Laurent Moret-Bailly, C\'edric P\'epin, Simon Riche and
Angelo Vistoli.

This work was supported by the ANR project CLap-CLap
(ANR-18-CE40-0026-01) and by the Centre Henri Lebesgue
(ANR-11-LABX-0020-01). I would like to thank the
executive and administrative staff of IRMAR and of the
Centre Henri Lebesgue for creating an attractive mathematical
environment.
\end{notitle*}

\tableofcontents

\section{Homomorphisms from a diagonalizable group}
\label{section:2}

The proof of Theorem~\ref{th-2} builds on the key case
where $G$ is a diagonalizable group scheme $D(M)$.
In this section, we establish representability in that case.
In Subsection~\ref{ssection:statement-and-reduction}
we state the result we want to prove and we reduce it to
the more specific statement~\ref{simplified-situation}.
In Subsection~\ref{section:3} we prove a crucial descent
statement used in Subsection~\ref{sec:repres-using-groth}
to complete the proof of~\ref{simplified-situation}
by verifying the conditions of Grothendieck's theorem on
unramified functors.

\subsection{Statement and first reductions}
\label{ssection:statement-and-reduction}

\begin{theorem} \label{theo:Hom_for_mult_groups}
Let $G,H$ be finitely presented $S$-group schemes with
$G$ diagonalizable and $H$ affine. Then $\Hom(G,H)$ is
representable by an $S$-scheme separated and locally of
finite presentation.
\end{theorem}

\begin{proof}
The group $G$ is a product $G=N\times \GG_m^r$ where~$N$
is finite diagonalizable. For a product $G=G_1\times G_2$,
the functor $\Hom(G,H)$ is the subfunctor of
$\Hom(G_1,H)\times \Hom(G_2,H)$ composed of pairs of maps
that commute. Using \cite{SGA3.2}, Exp.~VIII, 6.5.b) we
see that this is a closed subfunctor. It follows
that if the theorem is true for $G_1$ and $G_2$ then it is
true for $G$, hence it is enough to consider the factors
individually. If $G=N$ is finite, it is classical and
recalled in Lemma~\ref{lemma:Hom_proper_separated} that
$\Hom(N,H)$ is representable by an affine $S$-scheme.
It remains to handle the case $G=\GG_m$, which we now do.

The assumptions and conclusions of the theorem being local
for the Zariski topology on $S$, we can assume that $S$
is affine. Since $G$ and $H$ are of finite presentation,
with the usual results on limits (\cite{EGA}~IV, \S~8)
we see that $\Hom(G,H)\to S$ is locally of finite
presentation. Consequently we can further reduce to the
case where $S$ is of finite type over $\Spec(\ZZ)$.

For a prime number $\ell$ let $S_{\ell}\subset S$ be the
open subscheme where~$\ell$ is invertible. Choose
two distinct primes $\ell,\ell'$ and write
$S=S_{\ell}\cup S_{\ell'}$. Since the question of
representability is local on $S$, it is enough to handle
$S_{\ell}$ and $S_{\ell'}$ separately. In this way we
reduce to the case where $\ell\in \cO_S^\times$.

Let $\mu_{\ell^n}\subset \GG_m$ be the group scheme of
$\ell^n$-th roots
of unity. By Lemma~\ref{lemma:Hom_proper_separated} again,
the functor $\Hom(\mu_{\ell^n},H)$ is representable by an
affine $S$-scheme. In particular, the morphisms
$\Hom(\mu_{\ell^{n+1}},H)\to \Hom(\mu_{\ell^n},H)$ are
affine so the limit
\[
L\defeq\lim_n \Hom(\mu_{\ell^n},H)
\]
is representable by a scheme which is affine over $H$,
hence over $S$ also.
By restricting morphisms to the torsion
subschemes, we have a map of functors~:
\[
\varphi:\Hom(\GG_m,H) \too L,\quad
f\longmapsto \{f_{|\mu_{\ell^n}}\}_{n\ge 0}.
\]
It is enough to prove that $\varphi$ is representable by
schemes. For this let $T$ be an $S$-scheme and let $T\to L$
be a map, that is, a compatible collection
$\{u_n:\mu_{\ell^n,T}\to H_T\}$ of morphisms of $T$-group
schemes. We want to prove that the fibred product
$\Hom(\GG_m,H) \times_L T$ is representable. For this we
change our notation, rename $T$ as $S$, reduce to the
case where $T$ is affine as before, and the result is
exactly Theorem~\ref{simplified-situation} below.
\end{proof}

We have thus reduced the proof
of~\ref{theo:Hom_for_mult_groups} to the following
statement, whose proof occupies the rest of the section.

\begin{theorem} \label{simplified-situation}
Let $\ell$ be a prime number, $S=\Spec(R)$ an affine
$\ZZ[1/\ell]$-scheme of finite type, $H$ a finitely
presented affine $S$-group scheme, and
$\{u_n:\mu_{\ell^n}\to H\}_{n\ge 0}$ a family of
morphisms of $S$-group schemes such that $u_{n+1}$
extends $u_n$ for each $n$. Let $F$ be the functor
defined for all $S$-schemes $T$ by:
\[
F(T)=\big\{\mbox{morphisms of groups } f:\GG_{m,T}\to H_T
\mbox{ that extend the $u_{n,T}$, $n\ge 0$}\big\}.
\]
Then $F$ is representable by an $S$-scheme separated and
locally of finite presentation.
\end{theorem}

\subsection{Descent along schematically dominant morphisms}
\label{section:3}

We keep all notations as in~\ref{simplified-situation}.
The Density Theorem (\cite{SGA3.2}, Exp.~IX,
Th\'eor\`eme~4.7 and Remark~4.10) implies that $F(T)$
contains at most one point; that is, $F\to S$ is a
monomorphism.
To prove that~$F$ is representable, we will use
Grothendieck's theorem on unramified functors. The
verification that~$F$ fulfills the conditions of the
theorem will be based to a large extent on the following
fact: the map $F(T)\to F(T')$ is an isomorphism for all
schematically dominant morphisms of schemes $T'\to T$.
This is Lemma~\ref{lemma:descent-by-sch-dominant-2} below.
Its proof
will use a variation on the argument used to show that
formal homomorphisms from a group scheme of multiplicative type
to an affine group scheme are algebraic, see \cite{SGA3.2}
Exp.~IX, \S~7. It is the purpose of this subsection to
settle this.

We work over a $\ZZ[1/\ell]$-algebra $A$.

\begin{notitle}{$\ell$-power roots of unity}
\label{mu_ell_infty}
We consider the scheme of $\ell$-power roots of unity:
\[
\mu_{\ell^\infty}=\colim \mu_{\ell^n}.
\]
This is the disjoint sum of the schemes of {\em primitive}
roots of unity:
\[
\mu_{\ell^\infty}=
\underset{n\ge 0}{\coprod}\,\mu_{\ell^n}^*.
\]
If $\Phi_n$ denotes the $\ell^n$-th cyclotomic polynomial,
we have $A[\mu_{\ell^n}^*]=A[z]/(\Phi_n)$ and
\[
A[\mu_{\ell^\infty}]=\prod_{n\ge 0}A[z]/(\Phi_n).
\]
The restriction of functions is a canonical injective
morphism:
\[
c:A[\GG_m]\intoo A[\mu_{\ell^\infty}]
\]
which we describe further below.
\end{notitle}

\begin{notitle}{Cyclotomic expansion of Laurent polynomials}
\label{cyclotomic-expansion}
For relative integers $i\le j$ let
\[
A(i;j)=\bigg\{P=\sum_{i\le s\le j}a_sz^s \in A[z^{\pm 1}]
\bigg\}
\]
be the module of Laurent polynomials whose monomials
have degree in the range $\{i,\dots,j\}$.
\end{notitle}

\begin{lemma} \label{euclidean_division}
Each nonzero Laurent polynomial $P\in A[z^{\pm 1}]$
has a unique expression
\[
P=r_0+r_1(z-1)+r_2(z^\ell-1)+\dots+r_n(z^{\ell^{n-1}}-1)
\]
with $r_i\in A(-\lfloor \varphi(\ell^i)/2\rfloor;
\varphi(\ell^i)-\lfloor \varphi(\ell^i)/2\rfloor-1)$
and $r_n\ne 0$. In other words we have a decomposition
\[
A[z^{\pm 1}]=\bigoplus_{n\ge 0} D_n
\]
into sub-$A$-modules
$D_n\defeq A(-\lfloor \varphi(\ell^n)/2\rfloor;
\varphi(\ell^n)-\lfloor \varphi(\ell^n)/2\rfloor-1)
\cdot (z^{\ell^{n-1}}-1)$.
\end{lemma}

\begin{proof}
Let $\deg$ be the degree and $\val$ the valuation.
If $P$ is constant, the result is clear. Otherwise, there
is $n\ge 1$ minimal with the property that
\[
-\lfloor \varphi(\ell^n)/2\rfloor\le \val(P)\le \deg(P)
<\ell^n-\lfloor \varphi(\ell^n)/2\rfloor.
\]
Let $Q_0=z^{\lfloor \varphi(\ell^n)/2\rfloor}P$, so we have:
\[
0\le \val(Q_0)\le \deg(Q_0)<\ell^n.
\]
Let $B_n=A(-\lfloor \varphi(\ell^n)/2\rfloor;
\varphi(\ell^n)-\lfloor \varphi(\ell^n)/2\rfloor-1)
\cdot z^{\lfloor \varphi(\ell^n)/2\rfloor}$ be
the $z^{\lfloor \varphi(\ell^n)/2\rfloor}$-translate of the
$A$-module in the statement. For each $i$ the module
$B_i$ is finite free of rank
$\varphi(\ell^i)=\deg(\Phi_i)$, hence
\[
\rho_i:B_i\intoo A[z]\too A[z]/(\Phi_i)
\]
is an isomorphism and $B_i$ can serve as a module of
representatives of residue classes for Euclidean division
modulo $\Phi_i$. We define a sequence of polynomials
$\{Q_i\}$ by running the division algorithm:
\begin{itemize}
\item we set $s_0=\rho_0^{-1}(Q_0\modulo\Phi_0)$
and get a division $Q_0=s_0+\Phi_0Q_1$;
\item inductively, while $Q_i\ne 0$ we set
$s_i=\rho_i^{-1}(Q_i\modulo\Phi_i)$ and get a
division $Q_i=s_i+\Phi_iQ_{i+1}$.
\end{itemize}
Since the sequence $\{\deg(Q_i)\}_{i\ge 0}$ is strictly
decreasing, the process eventually stops. We obtain the
desired expression for~$P$ by setting
$r_i=z^{-\lfloor \varphi(\ell^n)/2\rfloor}s_i$.
\end{proof}

Like in the proof of Lemma~\ref{euclidean_division},
since $\text{rank}(D_i)=\varphi(\ell^n)$ the map
$D_n\into A[z^{\pm 1}]\to A[z^{\pm 1}]/(\Phi_n)$
is an isomorphism.

\begin{lemma} \label{lemma:map-c}
Let $c:A[\GG_m]\to A[\mu_{\ell^\infty}]=
\prod_{n\ge 0}A[z]/(\Phi_n)$,
$P\mapsto (P\modulo\Phi_0,P\modulo\Phi_1,
P\modulo\Phi_2,\dots)$ be the map of~\S\ref{mu_ell_infty}.
\begin{trivlist}
\itemn{1}
For an element $q_i\in A[z]/(\Phi_i)$, write $P_i$
the unique Laurent polynomial in $D_i$ with
$q_i=P_i\modulo\Phi_i$. The image of $c$
is the set of families $(q_0,q_1,q_2,\dots)$ such that there
is $N$ such that $q_n=P_0+P_1+\dots+P_N\modulo\Phi_n$
for all $n\ge N$. In this case, we have
$(q_0,q_1,q_2,\dots)=c(P)$ with $P=P_0+P_1+\dots+P_N$.
\itemn{2} If $A\to B$ is an injective ring homomorphism,
the commutative square of inclusions
\[
\begin{tikzcd}
A[\GG_m] \ar[r,hook] \ar[d,hook] & A[\mu_{\ell^\infty}]
\ar[d,hook] \\
B[\GG_m] \ar[r,hook] & B[\mu_{\ell^\infty}].
\end{tikzcd}
\]
is cartesian.
\end{trivlist}
\end{lemma}

\begin{proof}
(1) If $(q_0,q_1,q_2,\dots)=c(P)$, we can write
$P=P_0+P_1+\dots+P_N$ with $P_i\in D_i$. Since
$D_i\subset (z^{\ell^i}-1)A[z^{\pm 1}]$ and
$\Phi_i$ divides $z^{\ell^i}-1$, we have
$P_i=0\modulo\Phi_j$ for all $j\ge i$. Therefore
for all $n\ge N$ we have:
\[
q_n=P\modulo\Phi_n=P_0+P_1+\dots+P_N\modulo\Phi_n.
\]
Conversely if $q_n=P_0+P_1+\dots+P_N\modulo\Phi_n$
for all $n\ge N$ where $P_i\in D_i$ reduces to $q_i$
modulo $\Phi_i$, then obviously $(q_0,q_1,q_2,\dots)=c(P)$
with $P=P_0+P_1+\dots+P_N$.

\smallskip

\noindent {\rm (2)} This follows from the description
in~(1).
\end{proof}

We can now prove that the objects of the functor $F$ descend
along schematically dominant morphisms. For the latter notion,
we refer the reader to \cite{EGA} IV$_3$.11.10. The next
two lemmas are two variants of this descent statement.

\begin{lemma} \label{lemma:descent-by-sch-dominant-1}
Let $T'\to T$ be a morphism of $S$-schemes such that
$T=\Spec(A)$ is affine and $T'=\amalg_i \Spec(A_i)$
is a disjoint sum of affines, with $A\to \prod_i A_i$ injective.
Then the map $F(T)\to F(T')$ is bijective.
\end{lemma}

The result is easier when $T'\to T$ is quasi-compact,
but the general case will be crucial for us.

\begin{proof}
Since $F(T)$ has at most one point, the map $F(T)\to F(T')$ is
injective and it is enough to prove that it is surjective.
We start with an element of $F(T')$, i.e. a family of morphisms
of $A_i$-group schemes $f_i:\GG_{m,A_i}\to H_{A_i}$ each of which
extends the morphisms $u_n:\mu_{n,A_i}\to H_{A_i}$, $n\ge 0$.
For simplicity, in the sequel we write again
$f_i:\cO_H\otimes A_i\to A_i[z^{\pm 1}]$ and
$u_n:\cO_H\to R[z]/(z^{\ell^n}-1)$ the corresponding comorphisms
of Hopf algebras; this should not cause confusion. For each
$R$-algebra $A$ we also write
$u_{\infty,A}:\cO_H\otimes A\to A[\mu_{\ell^{\infty}}]$ for
the product of the $u_{n,A}$. These fit in a commutative diagram:
\[
\begin{tikzcd}[column sep=30]
\cO_H\otimes A_i \ar[d,"f_i"{swap}] \ar[rd,"u_{\infty,A_i}"]
& & \\
A_i[z^{\pm 1}] \ar[r,"c_{A_i}"{swap}] &
A_i[\mu_{\ell^\infty}].
\end{tikzcd}
\]
We now reduce to the case where $A$ and the $A_i$ are
noetherian. For this let $L$ resp. $L_i$ be the image of
$R\to A$, resp. of $R\to A_i$. Being quotients of $R$,
the rings $L$ and $L_i$ are noetherian. Moreover, since
$A\to \prod A_i$ is injective then so is $L\to \prod L_i$.
Since the $u_n$ are defined over $R$ hence over $L$,
we have a commutative diagram:
\[
\begin{tikzcd}[column sep=30,row sep=30]
\cO_H\otimes L_i \ar[d] \ar[rrd,"u_{\infty,L_i}"]
\ar[rd,"{f_{i,L_i}}"{swap} label distance=1mm,dotted] & \\
\cO_H\otimes L_i \ar[rd,"f_i"{swap}] &
L_i[z^{\pm 1}] \ar[r,"c_{L_i}"{swap}] \ar[d] &
L_i[\mu_{\ell^\infty}] \ar[d] \\
& A_i[z^{\pm 1}] \ar[r,"c_{A_i}"{swap}] &
A_i[\mu_{\ell^\infty}].
\end{tikzcd}
\]
Since the lower right square is cartesian, there is an induced
dotted arrow. In this way we see that $f_i$ is actually defined
over $L_i$. So replacing $A$ (resp. $A_i$) by $L$ (resp.
$L_i$), we obtain the desired reduction.

Let $\hat A\defeq \prod_i A_i$. Taking products over $i$,
we build a commutative diagram:
\[
\begin{tikzcd}[column sep=40,row sep=10]
\cO_H\otimes A \ar[d] \ar[rr,"u_{\infty,A}"]
\ar[rdd,dotted] &
& A[\mu_{\ell^\infty}] \ar[dd,hook] \\
\cO_H\otimes \hat A \ar[d] & & \\
\prod_i(\cO_H\otimes A_i) \ar[r,"\prod f_i"]
& \prod_i(A_i[z^{\pm 1}]) \ar[r,hook,"{\prod c_{A_i}}"]
& \hat A[\mu_{\ell^\infty}]
\end{tikzcd}
\]
Henceforth we set $C\defeq \cO_H\otimes A$ and we write
$\Phi_0$ the dotted composition in the diagram above.
What the diagram shows is that
\[
\begin{tikzcd}[column sep=45]
C \ar[r,"\Phi_0"] &
\prod_i(A_i[z^{\pm 1}]) \ar[r,hook,"{\prod c_{A_i}}"]
& \hat A[\mu_{\ell^\infty}]
\end{tikzcd}
\]
factors through $A[\mu_{\ell^\infty}]$. According to
Lemma~\ref{lemma:map-c}(2) applied with
$B=\hat A$, this implies that
\[
\begin{tikzcd}[column sep=45]
C \ar[r,"\Phi_0"] &
\prod_i(A_i[z^{\pm 1}]) \ar[r,hook,"{\prod \can_{A_i}}"]
& {\hat A}{}^{\,\ZZ}
\end{tikzcd}
\]
factors through $A^{\ZZ}$, providing a map
$\Phi:C \to A^{\ZZ}$.
From the diagrams expressing the fact that the~$f_i$
respect the comultiplications, taking products over $i$,
we obtain a commutative diagram:
\[
\begin{tikzcd}
C \ar[rr,"\Phi"] \ar[d] & & A^{\ZZ} \ar[d] \\
C\otimes_A C \ar[r,"\Phi\otimes\Phi"] &
A^{\ZZ}\otimes_A A^{\ZZ} \ar[r] & A^{\ZZ\times\ZZ}.
\end{tikzcd}
\]
Let $g\in C$ and write $\Phi(g)=(a_m)_{m\in\ZZ}$. Since
$A$ is noetherian, Lemme~7.2 of \cite{SGA3.2}, Exp.~IX
is applicable and shows that only finitely many of the
$a_m$ are nonzero, that is $\Phi(g)\in A[z^{\pm 1}]$.
Therefore $\Phi$ gives rise to a map $f:C\to A[z^{\pm 1}]$.
The fact that~$f$ respects the comultiplication of the
Hopf algebras follows immediately by embedding
$A[z^{\pm 1}]\otimes A[z^{\pm 1}]$ into
${\hat A}[z^{\pm 1}]\otimes {\hat A}[z^{\pm 1}]$ where the
required commutativity holds by assumption.
The fact that $f$ respects the counits is equally clear.
\end{proof}

\begin{lemma} \label{lemma:descent-by-sch-dominant-2}
Let $T'\to T$ be a morphism of $S$-schemes which is
schematically dominant. Then the map $F(T)\to F(T')$ is bijective.
\end{lemma}

\begin{proof}
Since $F(T)$ has at most one point, the map $F(T)\to F(T')$ is
injective and it is enough to prove that it is surjective.
By fpqc descent of morphisms, this holds when $T'\to T$
is a covering for the fpqc topology. Applying this remark
with chosen Zariski covers $\amalg T_i\to T$ and
$\amalg_{i,j} T'_{ij}\to \amalg_i T_i\times_T T'\to T'$,
we see that the vertical
maps in the following commutative square are bijective:
\[
\begin{tikzcd}
F(T) \ar[r]
\ar[d,"\sim"{above,label distance=1pt,sloped,pos=.4}] & F(T') \ar[d,"\sim"{above,label distance=1pt,sloped,pos=.4}] \\
\prod F(T_i) \ar[r] & \prod F(T'_{ij}).
\end{tikzcd}
\]
Therefore it is enough to prove that
$F(T_i)\to\prod_j F(T'_{ij})$ is bijective, for each $i$.
Choosing $T_i=\Spec(A_i)$ and $T'_{ij}=\Spec(A'_{ij})$ afffine,
the assumption that $T'\to T$ is schematically dominant implies
that $A_i\to \prod A'_{ij}$ is injective. In this way we are
reduced to the statement of
Lemma~\ref{lemma:descent-by-sch-dominant-1}.
\end{proof}

\subsection{Representability using Grothendieck's theorem
on unramified functors}
\label{sec:repres-using-groth}

Recall the statement of Grothendieck's theorem on representation
of unramified functors; all affine schemes $\Spec(A)$ appearing
are assumed to be $S$-schemes, and we write $F(A)$ instead
of $F(\Spec(A))$.

\begin{theorem-quote}{Grothendieck~\cite{Mu65}}
\label{Grothendieck-s-theorem}
Let $S$ be a locally noetherian scheme and $F$ a set-valued
contravariant functor on the category of $S$-schemes. Then $F$
is representable by an $S$-scheme which is locally of finite
type, unramified and separated if and only if Conditions
$(F_1)$ to $(F_8)$ below hold.
\begin{itemize}
\item[$(F_1)$] The functor $F$ is a sheaf for the fpqc topology.
\item[$(F_2)$] The functor $F$ is locally of finite presentation;
that is,
for all filtering colimits of rings $A=\colim A_\alpha$,
the map $\colim F(A_\alpha)\to F(A)$ is bijective.
\item[$(F_3)$] The functor $F$ is effective; that is, for all noetherian
complete local rings $(A,m)$, the map $F(A)\to \lim F(A/m^k)$ is
bijective.
\item[$(F_4)$] The functor $F$ is homogeneous; more precisely, for all exact
sequences of rings $A\to A'\toto A'\otimes_AA'$
with $A$ local artinian, $\length_A(A'/A)=1$ and trivial residue
field extension $k_A=k_{A'}$, the diagram
$F(A)\to F(A')\toto F(A'\otimes_AA')$ is exact.
\item[$(F_5)$] The functor $F$ is formally unramified.
\item[$(F_6)$] The functor $F$ is separated; that is, it satisfies the valuative
criterion of separation.
\end{itemize}
For the last two conditions we let $A$ be a noetherian ring,
$N$ its nilradical, $I$ a nilpotent ideal such that $IN=0$,
$T=\Spec(A)$, $T'=\Spec(A/I)$. We assume that $T$
is irreducible and we call $t$ its generic point.
\begin{itemize}
\item[$(F_7)$] Assume moreover that $A$ is complete one-dimensional local
with a unique associated prime. Then any point $\xi':\Spec(A/I)\to F$ such
that
\[
\xi'_t:\Spec((A/I)_t)\too \Spec(A/I)\too F
\]
can be lifted to a point $\xi^*:\Spec(A_t)\to F$, can be lifted
to a point $\xi:\Spec(A)\to F$.
\item[$(F_8)$] Assume that $\xi':\Spec(A/I)\to F$ is such that
\[
\xi'_t:\Spec((A/I)_t)\too \Spec(A/I)\too F
\]
can not be lifted to any subscheme of $\Spec(A_t)$ which is strictly
larger than $\Spec((A/I)_t)$. Then there exists a nonempty open set
$W\subset T$ such that for all open subschemes $W_1\subset T$ contained
in~$W$, the restriction $\xi'_{|W'_1}:W'_1\to F$ (with $W'_1=W_1\times_TT'$)
can not be lifted to any subscheme of~$W_1$ which is strictly larger than
$W'_1$.
\end{itemize}
\end{theorem-quote}

\noindent In what follows we apply Grothendieck's theorem
to prove~\ref{simplified-situation}, whose notation we use.
Recall in particular that
\[
F(T)=\{\mbox{morphisms } f:\GG_{m,T}\to H_T
\mbox{ that extend the $u_{n,T}$, $n\ge 0$}\}.
\]
We verify the conditions one by one for this functor.

\begin{notitle}{Conditions $(F_1)$, $(F_4)$, $(F_5)$,
$(F_6)$, $(F_7)$} We begin by checking the easiest conditions.

\medskip

\noindent $(F_1)$ This follows from fpqc descent, see e.g.
\cite{SGA1}, Exp.~VIII, Th.~5.2.

\smallskip

\noindent $(F_4)$ Since $A\to A'$ is injective,
by Lemma~\ref{lemma:descent-by-sch-dominant-1}
the map $F(A)\to F(A')$ is a bijection. This gives a statement
which is much stronger than the $(F_4)$ in the theorem.

\smallskip

\noindent $(F_5)$ Since $F\to S$ is a monomorphism, it is
formally unramified.

\smallskip

\noindent $(F_6)$ Since $F\to S$ is a monomorphism, it is separated.

\smallskip

\noindent $(F_7)$ Since $A$ has a unique associated prime, the map
$\Spec(A_t)\to\Spec(A)$ is schematically dominant.
Hence by Lemma~\ref{lemma:descent-by-sch-dominant-1}, the point
$\xi^*:\Spec(A_t)\to F$ automatically extends to $\Spec(A)$.
\end{notitle}

\begin{notitle}{Condition $(F_2)$} \label{condition_F2}
Let $\sF:=\Hom(\GG_m,H)$ be the functor of {\em all} morphisms
of group schemes $\GG_m\to H$, that is, not just those that
extend the collection $u_n$. It is standard that $\sF$ is
locally of finite presentation, see \cite{EGA} IV$_3$.8.8.3.
Should the affine scheme $\lim_n \Hom(\mu_{\ell^n},H)$ be
locally of finite type over $S$, it would follow that $F\to S$
is locally of finite presentation (\cite{EGA} IV$_1$.1.4.3(v)).
However this is not the case in general, and the verification
of $(F_2)$ needs more work.

So let $A=\colim A_\alpha$ be a filtering colimits of rings.
We want to prove that $\colim F(A_\alpha)\to F(A)$ is bijective.
We look at the diagram
\[
\begin{tikzcd}
\colim F(A_\alpha) \ar[r] \ar[d,hook] & F(A) \ar[d,hook] \\
\colim \sF(A_\alpha) \ar[r,"\sim"] & \sF(A).
\end{tikzcd}
\]
Since $\sF$ is locally of finite presentation, the bottom
row is an isomorphism. We deduce that the upper row is
injective. We shall now prove that the upper row is
surjective, and in fact that the diagram is cartesian.
\end{notitle}

\begin{lemma}
Let $f:\GG_{m,A}\to H_A$ be a morphism extending
$u_{n,A}:\mu_{\ell^n,A}\to H_A$
for all $n\ge 0$. Then there exists an index $\alpha$ such
that $f$ descends to a map $f_\alpha:\GG_{m,A_\alpha}\to H_{A_\alpha}$
extending $u_{n,A_\alpha}$ for all $n\ge 0$.
\end{lemma}

\begin{proof}
Since $G$ and $H$ are finitely presented, the morphism $f$
is defined at finite level, that is there exists an index $\alpha$
and a morphism of $A_\alpha$-group schemes
$g:\GG_{m,A_\alpha}\to H_{A_\alpha}$
whose pullback along $\Spec(A)\to\Spec(A_\alpha)$ is $f$.
Since the groups are affine, the morphism $g$ is given by a map of
rings $g^\sharp:\cO_H\otimes A_\alpha\to A_\alpha[z^{\pm 1}]$. Fix
a presentation $\cO_H=R[x_1,\dots,x_s]/(P_1,\dots,P_t)$. Then~:
\begin{itemize}
\item $u_n^\sharp$ is determined by the
elements $z_{n,j}\defeq u_n^\sharp(x_j)\in R[z]/(z^{\ell^n}-1)$
satisfying $P_k(z_{n,1},\dots,z_{n,s})=0$ for $k=1,\dots,t$,
\item $g^\sharp$ is determined by the elements
$y_j=g^\sharp(x_j)\in A_\alpha[z^{\pm 1}]$ satisfying
$P_k(y_1,\dots,y_s)=0$ for $k=1,\dots,t$.
\end{itemize}
Moreover, to say that $f$ extends $u_{n,A}$ means that
we have the equality $z_{n,j}=\pi_n(y_j)$ in
$A[z]/(z^{\ell^n}-1)$, for all $j$, where
\[
\pi_n:A_\alpha[z^{\pm 1}]\to A[z]/(z^{\ell^n}-1)
\]
is the projection. So we have to prove that we may enlarge
the index $\alpha$ in such a way that $g$ extends
$u_{n,A_\alpha}$ for all $n\ge 0$.

For $n_0\ge 0$ an integer, consider the finite free $R$-module
$E_0=R(-\lfloor\ell^{n_0}/2\rfloor;\ell^{n_0}-\lfloor\ell^{n_0}/2\rfloor-1)$ in the notation of
\ref{cyclotomic-expansion}. Then $\pi_{n_0|E_0}$
is an isomorphism and for all $n\ge n_0$ we can define
\[
\chi_n=\pi_n\circ (\pi_{n_0|E_0})^{-1}:
R[z]/(z^{\ell^{n_0}}-1) \too R[z]/(z^{\ell^n}-1).
\]
By base change, these objects are defined over any $R$-algebra.
We choose $n_0$ large enough so that $E_0\otimes_R A_\alpha$
contains the Laurent polynomials $y_1,\dots,y_s$.

In the present context, the condition that $f$ extends all the
maps $u_{n,A}:G_{n,A}\to H_A$ is a finiteness constraint imposed
by $f$ on $\{u_n\}$ (whereas in other places of our arguments
it is best seen as a condition imposed by $\{u_n\}$ on $f$).
Indeed, from the relations $z_{n,j}=\pi_n(y_j)$ in $A$, we deduce
that
\[
z_{n,j}=\chi_n(z_{n_0,j}) \quad\mbox{in }
A[z]/(z^{\ell^n}-1) \quad\mbox{for all } n\ge n_0,
\]
namely
$\chi_n(z_{n_0,j})=(\pi_n\circ (\pi_{n_0|E_0})^{-1})(\pi_{n_0}(y_j))
=\pi_n(y_j)=z_{n,j}$. We claim that we may increase $\alpha$
to achieve that these equalities hold in $A_\alpha[z]/(z^{\ell^n}-1)$,
for all $n\ge n_0$ and all $j$. In order to see this, note that
the elements $\delta_{n,j}:=z_{n,j}-\chi_n(z_{n_0,j})$ are defined
over $R$, and as we have just proved, they belong to the kernel of
the morphism $R[z]/(z^{\ell^n}-1)\to A[z]/(z^{\ell^n}-1)$. Let
$I\subset R$ be the ideal generated by the coefficients of the
expressions of $\delta_{n,j}$ on the monomial basis, for varying
$n\ge n_0$ and $j$. Since $R$ is noetherian, $I$ is generated by
finitely many elements. These elements vanish in $A$, hence they
vanish in~$A_\alpha$ provided we increase $\alpha$ a little, whence
our claim.

\smallskip

The relations $z_{n,j}=\pi_n(y_j)$ in $A[z]/(z^{\ell^n}-1)$
with $j=1,\dots,s$ and $n\le n_0$ being finite in number, we may
increase $\alpha$ so as to ensure that all of them hold in
$A_{\alpha}[z]/(z^{\ell^{n}}-1)$.
Then for $n\ge n_0$ we have
\[
z_{n,j}=\chi_n(z_{n_0,j})=\chi_n(\pi_{n_0}(y_j))=\pi_n(y_j)
\quad\mbox{in } A_{\alpha}[z]/(z^{\ell^{n_0}}-1)
\]
again. That is, $g$ extends
the maps $u_{n,A_{\alpha}}$ for all $n\ge 0$.
\end{proof}

\begin{notitle}{Condition $(F_3)$}
Let $(A,m)$ be a noetherian complete local ring. We want to prove
that the map $F(A)\to \lim F(A/m^k)$ is bijective. We write again
$\sF:=\Hom(\GG_m,H)$. We look at the diagram
\[
\begin{tikzcd}
F(A) \ar[r] \ar[d,hook] & \lim F(A/m^k) \ar[d,hook] \\
\sF(A) \ar[r,"\sim"] & \lim \sF(A/m^k).
\end{tikzcd}
\]
From \cite{SGA3.2}, Exp.~IX, Th.~7.1 we know that $\sF$ is effective,
that is the bottom arrow is bijective. We deduce that
the upper row is injective. We shall now prove that the upper row is
surjective, and in fact that the diagram is cartesian. So let
$f_k:\GG_{m,A/m^k}\to H_{A/m^k}$ be a collection of $A/m^k$-morphisms
such that $f_k$ extends $u_{n,A/m^k}:\mu_{\ell^n,A/m^k}\to H_{A/m^k}$
for all $n\ge 0$, and let $f:\GG_{m,A}\to H_A$ be a morphism that
algebraizes the $f_k$. We must prove that $f$ extends $u_{n,A}$,
for each $n$. For this let $i_n:\mu_{\ell^n,A}\to\GG_{m,A}$ be
the closed immersion. The two maps $f\circ i_n$ and $u_n$
coincide modulo $m^k$ for each $k\ge 1$, hence so do the
morphisms of Hopf algebras
\[
(f\circ i_n)^\sharp,u_n^\sharp:
\cO_H\otimes A\to A[z]/(z^{\ell^n}-1).
\]
Since $A[z]/(z^{\ell^n}-1)$ is separated for the $m$-adic
topology, we deduce that $(f\circ i_n)^\sharp=u_n^\sharp$
and hence $f\circ i_n=u_n$. This concludes the argument.
\end{notitle}

\begin{remark}
We could also appeal to the following more general result
extending the injectivity part of \cite{EGA}~III$_1$.5.4.1:
{\em let $(A,m)$ be a
noetherian complete local ring, and $S=\Spec(A)$. Let $X,Y$
be $S$-schemes of finite type with $X$ pure and $Y$ separated.
Let $f,g:X\to Y$ be $S$-morphisms. If we have the equality of
completions $\hat f=\hat g$, then $f=g$.} For the notion of
a pure morphism of schemes we refer to
Appendix~\ref{appendix:weil_restriction}. For the proof
of the italicized statement,
by \cite{EGA1new}, 10.9.4 the morphisms~$f$ and $g$ agree in
an open neighbourhood of $\Spec(A/m)$. Then the arguments in
the proof of \cite{Ro12}, Lemma~2.1.9 apply verbatim.
\end{remark}

\begin{notitle}{Condition $(F_8)$}
This condition will be verified with the help
of the following lemma.
\end{notitle}

\begin{lemma} \label{lemma-maximal-closed}
Let $T$ be a scheme and $T'$ a closed subscheme. Let
$\xi':T'\to F$ be a point. Then there is a largest closed
subscheme $\cZ_T\subset T$ such that $\xi'$ extends to $\cZ_T$.
Moreover, its formation is Zariski local: if $U\subset T$
is an open subscheme and $U'=U\cap T'$, we have $\cZ_T\cap U=\cZ_U$.
\end{lemma}

\begin{proof}
Throughout, for all open subschemes $U\subset T$ we write
$U'=U\cap T'$ and all closed subschemes $Z$ of $U$ such that
$\xi'_{|U'}$ extends to $Z$ are implicitly assumed to contain $U'$.
We proceed by steps.

Let $U=\Spec(A)$ be an affine open subscheme of $T$. Consider the
family of all closed subschemes $Z_\alpha=V(I_\alpha)\subset U$
to which $\xi'_{|U'}$ extends. Consider the ideal $I=\cap I_\alpha$
and define $\cZ_U=V(I)$. Since the map $A/I\to \prod A/I_\alpha$ is
injective, applying Lemma~\ref{lemma:descent-by-sch-dominant-1}, we
see that $\xi'_{|U'}$ extends to $\cZ_U$. By its very definition
the closed subscheme $\cZ_U$ is largest.

Let $U,V$ be two affine opens of $T$ with $U\subset V$. We claim
that $\cZ_V\cap U=\cZ_U$. Indeed, since $\xi'_{|V'}$ extends to $\cZ_V$
then $\xi'_{|U'}$ extends to $\cZ_V\cap U$, hence $\cZ_V\cap U\subset \cZ_U$.
Conversely, let $Z$ be the schematic image of $\cZ_U\to U\to V$.
The latter map being quasi-compact, the map $\cZ_U\to Z$ is
schematically dominant. By Lemma~\ref{lemma:descent-by-sch-dominant-2}
it follows that $\xi'_{Z_U}$ extends to $Z$. By maximality this
forces $Z\subset \cZ_V$, hence $\cZ_U\subset \cZ_V\cap U$.

Let $U,V$ be arbitrary affine opens of $T$. We claim that
$\cZ_U\cap V=\cZ_V\cap U$. Indeed, by the previous step, for all affine
opens
$W\subset U\cap V$ we have $\cZ_U\cap V\cap W=\cZ_W=\cZ_V\cap U\cap W$.

Let $\cZ_T$ be the closed subscheme of $T$ obtained by gluing
the $\cZ_U$ when $U$ varies over all affine opens; thus $\cZ_T\cap U=\cZ_U$
by construction. Now $\xi'$ extends to $\cZ_T$, because $\xi'_{|U}$
extends to $\cZ_U$ for each $U$, and we can glue these extensions.
Moreover $\cZ_T$ is maximal with this property, because if $\xi'$
extends to some closed subscheme $Z\subset T$ then for each affine
$U$ the element $\xi'_{|U}$ extends to $Z\cap U$, hence
$Z\cap U\subset \cZ_U=\cZ_T\cap U$, hence $Z\subset \cZ_T$.

The fact that $\cZ_T\cap U=\cZ_U$ for all open subschemes $U$ follows
by restricting to affine opens.
\end{proof}

In order to verify Condition $(F_8)$, we write $T=\Spec(A)$
and $T'=\Spec(A/I)$. The assumption that
$\xi'_t:\Spec((A/I)_t)\to \Spec(A/I)\to F$ does not lift to any
subscheme of $\Spec(A_t)$ which is strictly larger than
$\Spec((A/I)_t)$ means that the inclusion $T'\subset\cZ_T$
is an equality at the generic point. It follows that
$T'\cap W=\cZ_T\cap W=\cZ_W$ for some open $W$.
Applying Lemma~\ref{lemma-maximal-closed} to variable opens
$W_1\subset W$, we obtain $T'\cap W_1=\cZ_T\cap W_1=\cZ_{W_1}$,
which shows that $W$ fulfills the required condition and
we are done. \hfill $\square$

\bigskip

This concludes the proof of Theorem~\ref{simplified-situation}
and thus also of Theorem~\ref{theo:Hom_for_mult_groups}.

\section{Homomorphisms from a group with reductive and
proper composition factors}

In this section we build on Theorem~\ref{theo:Hom_for_mult_groups}
to prove our main results on the functors of homomorphisms
of group schemes and functors of reductive subgroups:
to wit, Theorems~\ref{th-2} and~\ref{th-3} from the Introduction.

\subsection{Homomorphisms from a reductive group}
\label{subsection:functor_reductive_subgroups}

In this section $G$ and $H$ are $S$-group schemes.

\begin{lemma} \label{lemma:Mono_reductive_separated}
Assume that $G\to S$ is reductive and $H\to S$ is separated
and of finite presentation. Let
$\Hom(G,H)$ be the functor of morphisms of group schemes,
and $\Mono(G,H)$ the subfunctor of monomorphisms.
Then the following hold.
\begin{itemize}
\itemn{1} The inclusion $\Mono(G,H)\subset \Hom(G,H)$ is
representable by open immersions.
\itemn{2} If $G\to S$ has no geometric fibre of
characteristic $2$ containing a direct factor isomorphic
to $\SO_{2n+1}$ for some $n\ge 1$, then for each maximal
torus $T\subset G$ the following commutative
diagram is cartesian:
\[
\begin{tikzcd}
\Mono(G,H) \ar[r,hook] \ar[d] & \Hom(G,H) \ar[d] \\
\Mono(T,H) \ar[r,hook] & \Hom(T,H).
\end{tikzcd}
\]
\end{itemize}
\end{lemma}

The refined statement~(2) will not be needed in the paper,
but we find it worth reporting.

\begin{proof}
We handle both cases~(1) and (2) simultaneously, with only
a little variation in the end.

The question of representability of
$\Mono(G,H)$ by an open subscheme of $\Hom(G,H)$ is \'etale-local
over $S$ so we may assume that $S$ is affine and that there
exists a maximal torus $T\subset G$ (\cite{SGA3.2}, Exp.~XII,
Th.~1.7). Also since $G$ and $H$ are finitely
presented, we may assume that $S$ is noetherian.

Let $f:G\to H$
and $K\defeq\ker(f)$. Restricting to the open locus where $f_{|T}$
is a monomorphism (\cite{SGA3.2}, Exp.~IX, Cor.~6.6)
we may also assume that $K\cap T=1$.

Let $S_0\subset S$ be the locus of points $s$ such that $K_s$
is trivial. We claim it is enough to prove that $S_0$
is open. Indeed, in this case $K\times_SS_0\to S_0$ is finite
and the augmentation ideal of its structure sheaf is zero in
each fibre, hence zero. It follows that $K\times_SS_0$ is
trivial and $S_0$ represents the subfunctor of $S$ defined
by the condition that~$f$ is a monomorphism.

We proceed to prove that $S_0$ is open in $S$. Since it is
constructible and $S$ is noetherian, it is enough to prove
that it is stable by generization, see \cite{EGA} IV$_3$.9.6.1.
Let $s\leadsto s_0$ be a specialization with $s_0\in S_0$ and
let $S'\to S$ be a morphism from a trait (spectrum of a
discrete valuation ring) whose image witnesses this
specialization. Let $S''\to S'$ be a ramified extension whose
residue field contains the field of definition of the geometric
nilpotent ideal of $K_{s''}$, so in particular
$(K_{s''})_{\red}$ is a smooth subgroup scheme of the
generic fibre $K_{s''}$. Replacing~$S$ by $S''$ we can assume
that $S$ is a trait and $(K_s)_{\red}$ is a smooth subgroup.
We have to prove that $K_s$ is trivial, assuming that
$K_{s_0}=1$ in case~(1) and that $K_{s_0}\cap T_{s_0}=1$
in case~(2).

Let $(K_s)^\circ_{\red}$ be the reduced identity component.
By the assumption $K\cap T=1$ together with conjugacy
of maximal tori, the normal, smooth connected subgroup
$(K_s)^\circ_{\red}$ contains no torus, hence is unipotent.
Since $G_s$ is reductive, we find $(K_s)^\circ_{\red}=1$.
This shows that $K_s$ is finite. If $L\subset K_s$ is a subgroup
of multiplicative type which is either \'etale or infinitesimal,
there is a maximal torus of $G_s$ containing $L$: in the \'etale
case this is standard, and in the infinitesimal case this is
Th.~1.1 of Geiss and Voigt \cite{GV04}. After a further \'etale
extension $S'/S$ if needed, a suitable conjugate of $L$ lies
in $K\cap T$, hence $L=1$. It follows that $K_s$ is finite
unipotent.

The neutral component $K_s^\circ$ is infinitesimal unipotent.
When $G\to S$ has no characteristic $2$ geometric fibre
containing a direct factor isomorphic to $\SO_{2n+1}$, the main
result of Vasiu~\cite{Va05} says that $G_s$ has no infinitesimal
unipotent group scheme, hence $K_s^\circ=1$.
In the general case, letting $G_d=\ker(F^d:G\to G^{(d)})$ be
the $d$-th Frobenius kernel of $G\to S$, we have
$K_s^\circ\subset (G_d)_s$ for large enough~$d$. Since $G_d$
is finite flat over $S$, the scheme-theoretic closure
$C\subset G_d$ of $K_s^\circ$ is finite flat also. Since $H$
is separated we have $C\subset K$ and from the assumption
$K_{s_0}=1$ we deduce that the rank of $C\to S$ is~1,
that is $C=1$. Thus $K^\circ_s=C_s=1$ and the
upshot is that $K_s$ is \'etale (and finite unipotent).

Then the automorphism scheme $\Aut(K_s)$ is \'etale, so by
connectedness of $G_s$ the conjugation action $G_s\to\Aut(K_s)$
is trivial. Therefore $K_s$ is central, and since the center
of $G_s$ is a torus, it is trivial.
\end{proof}

\begin{remark}
Let $S$ be the spectrum of a field $k$ of characteristic 2
and $V=ke_0\oplus\dots\oplus ke_{2n}$ the standard vector
space of dimension $2n+1$. Let $G=\SO(q)$ be the
orthogonal group of the quadratic form
$q(x_0,\dots,x_{2n})=x_0^2+x_1x_{n+1}+\dots+x_nx_{2n}$.
The polarization $\psi(x,y):=q(x+y)-q(x)-q(y)$ is
alternating with kernel $V_0=ke_0$, inducing a
nondegenerate alternating form $\psi_0$ on $W_0=V/V_0$.
This gives rise to an isogeny
$f:\SO_{2n+1}=\SO(q)\to\Sp_{2n}=\Sp(\psi_0)$ whose
kernel is isomorphic to $\alpha_2^{2n}$, see \cite{Va05},
\S~2.1. Let $T\subset \SO_{2n+1}$ be a maximal torus.
The restriction of $f$ to $T$ is a monomorphism; in other
words, the commutative square of
Lemma~\ref{lemma:Mono_reductive_separated}
is not cartesian.
\end{remark}

\begin{lemma} \label{lemma:Isom_reductive}
Assume that $G\to S$ is reductive and $H\to S$ is separated,
of finite presentation, flat and pure. Let
$\Isom(G,H)$ be the functors of isomorphisms of group
schemes.
\begin{trivlist}
\itemn{1} The inclusion $\Isom(G,H)\subset \Mono(G,H)$ is
representable by closed immersions of finite presentation,
\itemn{2} The inclusion $\Isom(G,H)\subset \Hom(G,H)$ is
representable by immersions locally of finite presentation.
\end{trivlist}
\end{lemma}

\begin{proof}
(1) By \cite{SGA3.2}, Exp.~XVI, Cor.~1.5.a) any
monomorphism $f:G\to H$ is a closed immersion. Thus
$\Isom(G,H)\subset \Mono(G,H)$ is the subfunctor defined
by the condition that the surjective map of sheaves
$\cO_H\to f_*\cO_G$ is an isomorphism. It follows
from \cite{RG71}, Premi\`ere partie, Th.~(4.1.1) that this
condition is representable by closed immersions
of finite presentation.

\smallskip

\noindent (2) Follows from (1) and
Lemma~\ref{lemma:Mono_reductive_separated}.
\end{proof}

\begin{theorem} \label{theorem:Hom_reductive_affine}
Assume that $G\to S$ is reductive and $H\to S$ is affine
and of finite presentation. Then, the following hold.
\begin{trivlist}
\itemn{1} $\Hom(G,H)$ is representable by an $S$-scheme
separated and locally of finite presentation.
\itemn{2} Each subscheme (resp. closed subscheme) of
$\Hom(G,H)$ which is quasi-compact over $S$ is
quasi-affine (resp. affine) over $S$.
\end{trivlist}
\end{theorem}

\begin{proof}
The fact that $\Hom(G,H)\to S$ is locally of finite
presentation is checked using the usual results on limits
from \cite{EGA}~IV, \S~8. Hence it will be enough
to complete the following three steps:
\begin{trivlist}
\itemm{i} prove that $\Hom(G,H)$ is a separated $S$-algebraic
space;
\itemm{ii} prove (2) (with ``subspace'' replacing ``subscheme'');
\itemm{iii} deduce that $\Hom(G,H)$ is a scheme.
\end{trivlist}

\noindent (i)
This question is \'etale-local over $S$ so by \cite{SGA3.3},
Exp.~XIX, Th.~2.5 we may assume that there exists a split
maximal torus $T\subset G$. Then it follows from
\cite{SGA3.3}, Exp.~XXIV, Cor.~7.1.9 that the restriction
map $\Hom(G,H)\to \Hom(T,H)$ is representable and affine.
Since $\Hom(T,H)\to S$ is representable and separated
by Theorem~\ref{theo:Hom_for_mult_groups}, we
see that $\Hom(G,H)\to S$ is representable by a separated
$S$-algebraic space.

\medskip

\noindent (ii)
Let $Y\subset \Hom(G,H)$ be a subspace which is
quasi-compact over $S$. In order to show that $Y$ is
quasi-affine over $S$ we can afford an \'etale base change
$S'\to S$, so we can assume that there exists a split
maximal torus $T\subset G$. Let $T_n=\ker(n:T\to T)$ be
the finite flat torsion subschemes. The limit
\[
L=\lim \Hom(T_n,H)
\]
is an affine $S$-scheme. As in \cite{SGA3.2}~7.1.1 we choose
a pinning of $G$ with a set of simple roots $\Delta$ of
cardinality $n$, and associated elements
$u_\alpha\in U_\alpha^\times(S)$, $w_\alpha\in\Norm_G(T)(S)$
for $\alpha\in\Delta$ (see {\em loc. cit.} for the precise
definition of these elements). It follows from \cite{SGA3.2},
Proposition~7.1.2 that the morphism
\[
\begin{tikzcd}[column sep=20,row sep=0
    ,/tikz/column 1/.append style={anchor=base east}
    ,/tikz/column 2/.append style={anchor=base west}
    ]
\Hom(G,H) \ar[r] & \Hom(T,H)\times H^{2n} \\
f \ar[r,mapsto] &
\big(f_{|T},(f(u_\alpha),f(w_\alpha))_{\alpha\in \Delta}\big)
\end{tikzcd}
\]
is a closed immersion. Using the Density Theorem we have
a sequence of monomorphisms
\[
\begin{tikzcd}[column sep=20]
F:=\Hom(G,H) \ar[r,closed] & \Hom(T,H)\times H^{2n}
\ar[r,hook] & L':=L\times H^{2n}.
\end{tikzcd}
\]
Let $Z\subset L'$ be the schematic image of $Y\to L'$;
this is a closed subscheme, hence affine over $S$.
Since~$Y$ is quasi-compact and $L'$ is separated, then
$Y\to L'$ is quasi-compact. Thus the morphism $Y\to Z$
is schematically dominant (\cite{SP22},
Tag~\href{https://stacks.math.columbia.edu/tag/01R8}{01R8}).
It then follows from Lemma~\ref{lemma:descent-by-sch-dominant-2}
that $Y\to \Hom(T,H)$ factors uniquely through a map
$Z\to \Hom(T,H)$. Thus the closed immersion $Z\to L'$ factors
through a closed immersion $Z\to \Hom(T,H)\times H^{2n}$.
Then $Y\to Z$ is an immersion; in particular it is
quasi-affine so~$Y$ is quasi-affine over $S$. If at the
start it is assumed that $Y\subset \Hom(G,H)$ is an
$S$-quasi-compact {\em closed} subspace, then $Y\to Z$
is a closed immersion; being also schematically dominant,
it is an isomorphism, hence $Y$ is affine over $S$ in this
case.

\medskip

\noindent (iii)
Let $\Hom(G,H)=\cup Z_i$ be a covering by quasi-compact
open subspaces. According to (ii) the $Z_i$ are quasi-affine
over $S$. In particular they are schemes, hence
$\Hom(G,H)$ is a scheme.
\end{proof}

\subsection{Functor of reductive subgroups}

Recall that a {\em type} is
by definition an isomorphism class $\type=[\sR]$ of
{\em root datum} (\cite{SGA3.2}, XXII.2.6.1), the latter
being a quadruple $\sR=(M,M^*,R,R^*)$
composed of finite type free $\ZZ$-modules in duality
$M,M^*$ and finite subsets $R\subset M$, $R^*\subset M^*$
in duality called {\em root system} and {\em coroot system}
(\cite{SGA3.2}, XXI.1.1.1). For example, the type of a
diagonalizable group $T=D(M)$ is
$\type=[(M,M^*,\varnothing,\varnothing)]$.

\begin{theorem} \label{theorem:subgroups_reductive_affine}
Assume that $H\to S$ is an affine group scheme of finite
presentation.
\begin{trivlist}
\itemn{1} The functor $\Sub_{\red}(H)$ of reductive
subgroups of $H$ is representable by an algebraic space
separated and locally of finite presentation which is
a disjoint sum
\[
\Sub_{\red}(H)=\underset{\type}{\textstyle\coprod}
\,\Sub_{\type}(H)
\]
indexed by the types of reductive groups.
\itemn{2} The summand of subgroups {\em of multiplicative type}
\[
\Sub_{\mult}(H)=\underset{\type=[(M,M^*,\varnothing,\varnothing)]}{\textstyle\coprod}
\,\Sub_{\type}(H)
\]
is representable by a scheme with the following property:
each subscheme (resp. closed
subscheme) of $\Sub_{\mult}(H)$ which is quasi-compact over $S$,
is quasi-affine (resp. affine) over $S$.
\end{trivlist}
\end{theorem}

\begin{proof}
(1) The disjoint sum decomposition reflects the fact that
the type of a reductive group is locally constant
on the base (\cite{SGA3.2}, XXII.2.8). Thus it is enough to
establish that $\Sub_{\type}(H)$ is representable.

Let $G(\type)$ be the split reductive group scheme of type
$\type$ as in \cite{SGA3.3}, XXV.1.1.
By Theorem~\ref{theorem:Hom_reductive_affine}, the functor
$\Hom(G(\type),H)$ is representable by a scheme. It follows from
Lemma~\ref{lemma:Mono_reductive_separated}, item~(1) that the
subfunctor
\[
\Mono(G(\type),H)\subset\Hom(G(\type),H)
\]
of monomorphisms of group schemes is an open subscheme.
Moreover, since $G(\type)$ is reductive and $H$ is finitely
presented and separated, by \cite{SGA3.2}, XVI.1.5.a) any
monomorphism $f:G(\type)\to H$ is a closed immersion, inducing
an isomorphism between $G(\type)$ and a closed subgroup scheme
$K\into H$. By taking a monomorphism $f$ to its image $K$,
we obtain a morphism of functors:
\[
\pi:\Mono(G(\type),H)\too \Sub_{\type}(H).
\]
Let $A=\Aut(G(\type))$ be the functor of automorphisms
of $G(\type)$; this is a smooth, separated $S$-group scheme
by \cite{SGA3.3}, XXIV.1.3. It acts freely on
$\Mono(G(\type),H)$ by the rule $af=f\circ a^{-1}$ for
$a\in\Aut(G(\type))$ and $f\in\Mono(G(\type),H)$.
Let $\Mono(G(\type),H)/A$ be the quotient sheaf.
Since the morphism $\pi$  is $A$-equivariant, it induces
a morphism
\[
i:\Mono(G(\type),H)/A\too \Sub_{\type}(H).
\]
We claim that $i$ is an isomorphism. It is enough to
prove that it is an isomorphism of fppf sheaves:
\begin{itemize}
\item surjectivity: let $K\subset H$ be a reductive
subgroup scheme. Around each point $s\in S$, after \'etale
localization there is a maximal torus $T$ (\cite{SGA3.3},
XIX.2.5) and after further Zariski localization there is a
root system for $T$ providing a splitting for $K$
(\cite{SGA3.3}, XXII.2.1). So we can assume that $K$
is split and isomorphic to $G(\type)$.
We obtain a monomorphism $G(\type)\simeq K\subset H$ which
provides a point of $\Mono(G(\type),H)$ lifting $K$.
\item injectivity: if $f_i:G(\type)\to H$ are two
monomorphisms with the same image $K$, then
$f_2^{-1}\circ f_1:G(\type)\to K\to G(\type)$ is an
automorphism of $G(\type)$.
\end{itemize}
Now by Artin's Theorem (see \cite{SP22},
Tag~\href{https://stacks.math.columbia.edu/tag/04S5}{04S5})
the quotient $\Mono(G(\type),H)/A$ is an algebraic space
locally of finite presentation, hence so is $\Sub_{\type}(H)$.
Moreover, using that monomorphisms $G(\type)\to H$ are closed
immersions we see that $\Sub_{\type}(H)$ is separated. This
concludes the proof of (1).

\medskip

\noindent (2) Let us write $\Sub_M(H)\defeq\Sub_{\type}(H)$
when $\type=[(M,M^*,\varnothing,\varnothing)]$.
To prove that $\Sub_M(H)$ is representable by a scheme,
let $L:=\lim \Sub_{M/nM}(H)$ be the limit of the functors
of finite flat multiplicative type subgroups of type $M/nM$.
Since $\Sub_{M/nM}(H)$ is representable and affine
(Lemma~\ref{lemma:Hom_proper_separated}), the functor $L$
is an affine scheme. By mapping any subgroup $G\into H$ to
the collection of subgroups $G_n\into H$
where $G_n=\ker(n:G\to G)$, we define a morphism of functors
$u:\Sub_M(H)\to L$. By the Density Theorem, this is a
monomorphism. As $\Sub_M(H)\to S$ is locally of finite type,
so is $u$. In particular $u$ is a separated, locally
quasi-finite morphism. By \cite{SP22},
Tag~\href{https://stacks.math.columbia.edu/tag/0418}{0418}
all such morphisms are representable by schemes, hence
$\Sub_M(H)$ is a scheme. Finally, in order to prove that
each subscheme (resp. closed subscheme) which is quasi-compact
over $S$ is quasi-affine (resp. affine) over $S$, we proceed
as in the
proof of~\ref{theorem:Hom_reductive_affine}.
\end{proof}

\subsection{Homomorphisms from a group with reductive and
proper composition factors}

\begin{lemma} \label{lemma:extending_from_subgroup}
Let $1\to N\to G\to Q\to 1$ be an exact sequence of flat,
finitely presented $S$-group schemes with $N\to S$ pure.
Let $H$ be an affine, finitely presented $S$-group scheme
and $f_0:N\to H$ a morphism of group schemes. Assume that
\begin{trivlist}
\itemm{i} $Q$ is reductive, or
\itemm{ii} $Q$ is proper.
\end{trivlist}
Then the functor $\Hom^{f_0}(G,H)$ of morphisms of group
schemes $f:G\to H$ extending $f_0$ is representable by a
locally finitely presented, separated $S$-algebraic space.
\end{lemma}

\begin{proof}
Let $\Gamma_0\subset N\times H\subset G\times H$ be the graph
of $f_0$. Since $\Gamma_0\simeq N$ is pure, by
Corollary~\ref{coro:representability_normalizer}
its normalizer $\Norm(\Gamma_0)\subset G\times H$ is a
closed, finitely presented subgroup scheme of
$G\times H$. Let
$\pi:G\times H\to G$ be the projection and
$\pi':\Norm(\Gamma_0)/\Gamma_0\to G/N=Q$ the morphism
it induces.

\medskip

\noindent {\em Step 1: the map
$\pi':\Norm(\Gamma_0)/\Gamma_0\to Q$ is affine.}
Since the closed immersion $\Norm(\Gamma_0)\into G\times H$
induces a closed immersion of algebraic spaces
$\Norm(\Gamma_0)/\Gamma_0\into (G\times H)/\Gamma_0$ (beware
that the target does not a priori carry a group structure),
it is enough to prove that $(G\times H)/\Gamma_0\to Q$
is affine. It is enough to check this after the fppf base
change $G\to Q$. For this, we consider the morphism:
\[
\begin{tikzcd}[column sep=25, row sep = 0,
    ,/tikz/column 1/.append style={anchor=base east}
    ,/tikz/column 2/.append style={anchor=base west}
    ]
(G\times H)\times_Q G \ar[r,"\alpha"] & H\times G \\
((g_1,h),g_2) \ar[r,mapsto] & (f_0(g_2g_1^{-1})h,g_2)
\end{tikzcd}
\]
(note that $g_2g_1^{-1}$ is a section of $N$).
This is invariant by the action of $\Gamma_0$ on
$(G\times H)\times_Q G$ by translation on the first factor.
By commutation of the quotient $G\times H\to (G\times H)/\Gamma_0$
with the flat base change $G\to Q$, from $\alpha$ we deduce
a morphism
\[
((G\times H)/\Gamma_0)\times_Q G\stackrel{\beta}{\tooo} H\times G.
\]
It is easy to see that the map $(h,g)\mapsto (g,h,g)$ provides
an inverse to $\beta$ which therefore is an isomorphism.
Since the right-hand side is affine over $G$, this concludes
Step~1.

\medskip

\noindent {\em Step 2: conclusion.}
Attaching to a morphism $f:G\to H$ its graph $\Gamma$ yields
a correspondence between the functor $\Hom^{f_0}(G,H)$
on one side, and the functor of subgroups
$\Gamma\subset G\times H$ containing $\Gamma_0$ and such that
$\pi_{|\Gamma}:\Gamma\to G$ is an isomorphism, on the other
side. Note that $\Gamma\subset \Norm(\Gamma_0)$ because
$N$ is normal in~$G$; hence the latter functor is in
correspondence with the functor of subgroups~$\Gamma'$ of
$\Norm(\Gamma_0)/\Gamma_0$ such that
$\pi'_{|\Gamma'}:\Gamma'\to Q$ is an isomorphism.
It remains to prove that the latter is representable.

In case (i), by Step~1 the map
$\Norm(\Gamma_0)/\Gamma_0\to Q\to S$ is affine. By 
Theorem~\ref{theorem:subgroups_reductive_affine} the
functor of reductive subgroups of $\Norm(\Gamma_0)/\Gamma_0$
is representable. The subfunctor of those subgroups $\Gamma'$
for which $\pi'_{|\Gamma'}:\Gamma'\to Q$ is an isomorphism is
representable by a locally finitely presented subscheme by
Lemma~\ref{lemma:Isom_reductive}.

In case (ii) recall from
Lemma~\ref{lemma:functor_proper_subgroups} that the functor
of proper, flat, finitely presented subgroups of
$\Norm(\Gamma_0)/\Gamma_0$ is representable.
According to Olsson \cite{Ol06}, Lemma~5.2, the
subfunctor of those subgroups~$\Gamma'$ for which
$\pi'_{|\Gamma'}:\Gamma'\to Q$ is an isomorphism is
representable by an open subscheme.
\end{proof}

\begin{theorem} \label{theo:Hom_from_proper_reductive}
Assume that $G\to S$ has a finite composition series whose
factors are either reductive or proper, flat, finitely
presented, and that $H\to S$ is affine and of finite
presentation.
\begin{trivlist}
\itemn{1} $\Hom(G,H)$ is representable by an
$S$-algebraic space separated and locally of finite presentation.
\itemn{2} $\Mono(G,H)$ is representable by an open
subspace of $\Hom(G,H)$. Moreover, all monomorphisms
$G\to H$ are closed immersions.
\end{trivlist}
\end{theorem}

\begin{proof}
According to Corollaries~\ref{coro:connected_fibres_is_pure}
and~\ref{coro:extension_is_pure}, all group schemes having
a finite composition series as indicated are pure.
We prove (1) and (2) for group schemes admitting a composition
series of length~$n$, by induction on $n$. If $n=0$ we have
$G=1$ and both statements are clear, so we now assume
that~$G$ admits a composition series of length $n\ge 1$.
Thus there is an exact sequence $1\to N\to G\to Q\to 1$
where $Q$ is reductive or proper and $N$ admits a composition
series of length $n-1$.

\smallskip

\noindent (1) By induction the functor $\Hom(N,H)$
is representable by an $S$-algebraic space separated and
locally of finite presentation, so it is enough to prove that
the restriction homomorphism $\Hom(G,H)\to \Hom(N,H)$ is
representable, separated and locally of finite presentation.
This is exactly what
Lemma~\ref{lemma:extending_from_subgroup} says.

\smallskip

\noindent (2) By induction we have an open immersion
$\Mono(N,H)\subset\Hom(N,H)$ and it is enough to prove
that $\Mono(G,H)$ is open in the space
$T=\Hom(G,H)\times_{\Hom(N,H)}\Mono(N,H)$ of morphisms
whose restriction to $N$ is monomorphic. Let $f:G\to H$ be
the universal morphism over $T$ and let $N'=f(N)\simeq N$.
The normalizer $\Norm(N')\subset H$ is a closed, finitely
presented subgroup scheme of $H$ thanks to
Corollary~\ref{coro:representability_normalizer}. Moreover
$f$ maps into $\Norm(N')$ and induces a morphism
$f':Q=G/N\to \Norm(N')/N'$. Now $\Mono(G,H)$ is the subfunctor
of $T$ where $f'$ is a monomorphism, which is an open
subscheme by Lemma~\ref{lemma:Mono_reductive_separated}
when $Q$ is reductive and by
Lemmma~\ref{lemma:Hom_proper_separated} when $Q$ is proper.

Finally the fact that all monomorphisms $G\to H$ are
closed immersions follows directly from the same statements
for reductive and proper groups. For reductive groups this
is \cite{SGA3.2}, Exp.~XVI, Cor.~1.5.a) and for proper groups
this is because proper monomorphisms are closed immersions.
\end{proof}

\section{Algebraicity and smoothness of fixed points stacks}
\label{section:algebraicity-of-fixed-pts}

Let $\sX\to S$ be an algebraic stack and $G\to S$ a
group algebraic space acting on it. 

\subsection{Algebraicity}

The stack of fixed points $\sX^G$ has for sections over a
scheme $T\to S$ the pairs $(x,\{\alpha_g\}_{g\in G(T)})$
composed of an object $x\in \sX(T)$ and a collection of
isomorphisms $\alpha_g:gx\to x$ satisfying the cocycle
condition $\alpha_{gh}=\alpha_g\circ g\alpha_h$ (see
\cite{Ro05}, Prop.~2.5), pictured by a commutative triangle:
\[
\begin{tikzcd}[column sep={4em,between origins}]
& gx
\ar[rd,"\alpha_g",start anchor={[xshift=-.5ex]},end anchor={[xshift=-1ex]}] & \\
(gh)x \ar[ru,"g\alpha_h"] \ar[rr,"\alpha_{gh}",swap] & & x.
\end{tikzcd}
\]
An interesting viewpoint on $\sX^G$ is that it can be expressed
as a certain Weil restriction of the {\em universal stabilizer}
$\St_{\sX\!,G}$ of the action of $G$ on $\sX$. The latter is
the 2-fibred product:
\[
\begin{tikzcd}[column sep=2cm]
\St_{\sX\!,G} \ar[r] \ar[d] & \sX \ar[d,"\Delta"] \\
G\times\sX \ar[r,"\text{act}\times\pr_2"] & \sX\times\sX.
\end{tikzcd}
\]
In particular $\St_{\sX\!,G}\to\sX$ is representable by
algebraic spaces. The top horizontal map of the diagram makes
$\St_{\sX\!,G}$ an $\sX$-group functor:  for each
$x:T\to \sX$ we have:
\[
\St_{\sX\!,G}(T)=\big\{(g,\alpha); \ g\in G(T) \mbox{ and } 
\alpha:gx\isomto x \mbox{ an isomorphism}\big\},
\]
with law of multiplication
$(g,\alpha) \cdot (h,\beta) \defeq (gh,\alpha\circ g\beta)$
and neutral element $(g,\alpha)=(1,\id_x)$. The left vertical
map is the map $\St_{\sX,G}\to G_\sX$, $(g,\alpha)\mapsto g$.
It is a morphism of $\sX$-group spaces with kernel equal to
the inertia stack $I_\sX\defeq \sX\times_{\sX\times \sX}\sX$,
whence an exact sequence of $\sX$-group functors:
\[
1 \too I_\sX \too \St_{\sX\!,G} \too G_\sX.
\]
If the diagonal $\sX\to \sX\times\sX$ is affine and finitely
presented, then so is $\St_{\sX\!,G}\to G_\sX$.


\begin{definition}
In general, if $\Sigma\to G$ is a morphism of $S$-group spaces we
write
\[
(\Resgr_{G/S}\Sigma)(T)=\big\{
\mbox{group-theoretic sections of $\Sigma_T\to G_T$}\big\}.
\]
We call $\Resgr_{G/S}\Sigma$ the {\em group-theoretic Weil
restriction of $\Sigma$ along $G\to S$}.
\end{definition}

\begin{lemma}
We have an $\sX$-isomorphism
$\sX^G\isomto \Resgr_{G_\sX\!/\sX}\St_{\sX\!,G}$.
\end{lemma}

\begin{proof}
A section of $\sX^G$ over $x:T\to \sX$ is a collection
$(\{\alpha_g\}_{g\in G(T)})$ satisfying the cocycle
condition $\alpha_{gh}=\alpha_g\circ g\alpha_h$. This is exactly
a group-theoretic section of
$\St_{\sX\!,G}\times_{\sX}T\to G_\sX\times_{\sX}T$.
\end{proof}

For the proof of Theorem~\ref{th-fixed-pt-stack} below we
need a variant of Lemma~\ref{lemma:extending_from_subgroup}.

\begin{lemma} \label{lemma:extending_sections}
Let $1\to N\to G\to Q\to 1$ be an exact sequence of flat,
finitely presented $S$-group schemes with $N\to S$ pure.
Let $E$ be an $S$-group scheme and $\pi:E\to G$ a morphism of
$S$-group schemes which is affine. Let $f_0:N\to E$ be a
morphism of group schemes which is a section of
$\pi_{|N}:E\times_GN\to N$. Assume that
\begin{trivlist}
\itemm{i} $Q$ is reductive, or
\itemm{ii} $Q$ is proper.
\end{trivlist}
Then the functor $\Resgr_{G/S}^{f_0}(E)$ of group-theoretic
sections of $\pi$ extending $f_0$ is representable by a
locally finitely presented, separated $S$-algebraic space.
\end{lemma}

\begin{proof}
The section $f_0$ induces an isomorphism between $N$ and
the image subgroup $\Sigma_0\defeq f_0(N)$. Let
$\Norm(\Sigma_0)\subset E$ be its normalizer. Since
$\Sigma_0\simeq N$ is pure, by
Corollary~\ref{coro:representability_normalizer} this is a
closed, finitely presented subgroup scheme of $E$.
Let $\pi':\Norm(\Sigma_0)/\Sigma_0\to G/N=Q$ be the morphism
induced by $\pi:E\to G$.

\medskip

\noindent {\em Step 1: the map
$\pi':\Norm(\Sigma_0)/\Sigma_0\to Q$ is affine.}
Since $\Norm(\Sigma_0)/\Sigma_0\into E/\Sigma_0$ is a closed
immersion, it is enough to prove that $E/\Sigma_0\to Q$ is
affine. In turn, it is enough to prove this after the flat
base change $G\to Q$. The morphism
\[
\alpha:E\times_QG\too E,\quad
(e,g) \longmapsto (f_0(g\pi(e)^{-1})e
\]
is invariant by the action of $\Sigma_0$ on $E\times_Q G$
by translation on the first factor, hence induces
\[
(E/\Sigma_0)\times_Q G\stackrel{\beta}{\tooo} E.
\]
The map $e\mapsto (e,\pi(e))$ is an inverse to $\beta$ which
therefore is an isomorphism. Since the right-hand side is
affine over $G$, our claim follows.

\medskip

\noindent {\em Step 2: conclusion.} We have an isomorphism
of functors between $\Resgr_{G/S}^{f_0}(E)$ and the functor
of subgroups $\Sigma\subset E$ containing $\Sigma_0$ such
that $\pi_{|\Sigma}$ is an isomorphism. That functor
is isomorphic to the functor of subgroups~$\Sigma'$
of $\Norm(\Sigma_0)/\Sigma_0$ such that
$\pi'_{|\Sigma'}:\Sigma'\to Q$ is an isomorphism.
It remains to prove that the latter is representable.

In case (i), by Step~1 the composition
$\Norm(\Sigma_0)/\Sigma_0\to Q\to S$ is affine. By 
Theorem~\ref{theorem:subgroups_reductive_affine} the
functor of reductive subgroups of $\Norm(\Sigma_0)/\Sigma_0$
is representable. The subfunctor of those subgroups $\Sigma'$
for which $\pi'_{|\Sigma'}$ is an isomorphism is
representable by a locally finitely presented subscheme by
Lemma~\ref{lemma:Isom_reductive}.

In case (ii) recall from 
Lemma~\ref{lemma:functor_proper_subgroups} that the functor
of proper, flat, finitely presented subgroups of
$\Norm(\Sigma_0)/\Sigma_0$ is representable.
According to Olsson \cite{Ol06}, Lemma~5.2, the
subfunctor of those subgroups~$\Sigma'$ for which
$\pi'_{|\Sigma'}$ is an isomorphism is
representable by an open subspace.
\end{proof}

\begin{theorem} \label{th-fixed-pt-stack}
Let $\sX\to S$ be an algebraic stack with affine, finitely
presented diagonal and let $G\to S$ be a group space acting
on $\sX$. Assume that $G\to S$ has a finite composition series
whose factors are either reductive or proper, flat, finitely
presented. Then the fixed point stack $\sX^G\to S$ is algebraic,
and the morphism $\sX^G\to \sX$ is representable by algebraic
spaces, separated and locally of finite presentation.
If $G\to S$ is reductive, this morphism is even representable
by schemes.
\end{theorem}

\begin{proof}
According to Corollaries~\ref{coro:connected_fibres_is_pure}
and~\ref{coro:extension_is_pure}, all group schemes having
a finite composition series as indicated are pure. Therefore
the assumption implies that there is an exact sequence
\[
1\too N\too G\too Q\too 1
\]
of flat, finitely presented $S$-group schemes with $N\to S$
pure and $Q\to S$ either reductive or proper. By induction
on the number of factors in a composition series, it is enough
to prove that the map $\sX^G\to\sX^N$ is representable by
algebraic spaces separated and locally of finite presentation.
Let $\St\defeq \St_{\sX\!,G}$ be the universal stabilizer
of $G$ acting on $\sX$. Let $x:T\to\sX^N$ be a point from
an $S$-scheme; this corresponds to a group-theoretic section
$f_0:N\to \St$ of $\St\times_GN\to N$.
The functor $\sX^G\times_{\sX^N}T$ classifies the
group-theoretic sections of $\St\to G$ that extend $f_0$. By
Lemma~\ref{lemma:extending_sections}, this is representable
by an algebraic space enjoying the announced properties.
\end{proof}

Alper, Hall and Rydh proved this result in \cite{AHR20},
Theorem~5.16 when $\sX$ is a Deligne-Mumford stack locally
of finite type over a field, and $G=\GG_m$. They further
showed that in this situation $\sX^G\to\sX$ is a monomorphism;
this can be easily extended to the case where $\sX$ is a
Deligne-Mumford stack, $G$ is smooth with connected fibres,
and the base scheme $S$ is arbitrary. In the following
example, we show that the Deligne-Mumford assumption is
essential.

\begin{example} \label{example-B-alpha-p}
({\em An algebraic stack $\sX$ with finite
inertia with action of a torus $T$ such that
$\sX^T\to\sX$ is not a monomorphism}.)
Let $S$ be a scheme of characteristic $p>0$. Let
$\sX=B\alpha_p$ be the classifying stack of~$\alpha_p$.
Consider the exact sequence of commutative $S$-group schemes:
\[
0 \too \alpha_p \too \GG_a \stackrel{\Frob}{\too} \GG_a \too 0.
\]
The group $T:=\GG_m=\Aut(\GG_a)$ acts on the first and
second term naturally, and on the third term via Frobenius,
that is $\lambda\cdot x:=\lambda^px$. In this way the sequence
is an exact sequence of $T$-modules. There is an induced,
$T$-equivariant exact sequence of Picard categories:
\[
0 \too \alpha_p(S) \too \GG_a(S) \too \GG_a(S)
\too (B\alpha_p)(S) \too (B\GG_a)(S) \too (B\GG_a)(S)
\]
(see Giraud~\cite{Gi71}, Chap.~III, \S~3, Prop.~3.2.1).

We claim that $\sX^T\to\sX$ is not a monomorphism.
Indeed, if $S=\Spec(R)$ and $\lambda\in T(R)=R^\times$ then
any $\alpha_p$-torsor over $S$ is of the form
$P=\Spec(R[x]/(x^p-r))$ with $\alpha_p$-action given by
$a\cdot x=a+x$, for some $r\in R$. Moreover, the torsor
$\lambda\cdot P$ is $\Spec(R[x]/(x^p-\lambda^pr))$ with
action $a\cdot x=\lambda a+x$. Let us fix such a torsor
$P\to S$, that is, a map $S\to\sX$. The fibre product
$\sX^T\times_{\sX}S$ is the functor of $T$-linearizations
of $P$. This is identified with the functor of group-theoretic
sections of the extension
\begin{equation} \label{canonical-extension}
1 \too \alpha_p=\Aut(P) \too E \too \GG_m \too 1
\end{equation}
where $E=\{(\lambda,\varphi)\:;\:\lambda \in T
\mbox{ and } \varphi:\lambda\cdot P\to P
\mbox{ an isomorphism}\}$. One computes that all the
isomorphisms $\varphi:\lambda\cdot P\to P$ are described by
a map of algebras $R[x]/(x^p-r)\to R[x]/(x^p-\lambda^pr)$,
$x\mapsto u+\lambda^{-1}x$ for some $u\in\alpha_p(R)$.
Thus we see that $E$ is the group scheme whose points are pairs
$(\lambda,u)\in T\times\alpha_p$ with law of multiplication
\[
(\lambda,u)\cdot (\mu,v)=(\lambda\mu,u+\lambda^{-1}v).
\]
The sections of the extension (\ref{canonical-extension})
are the maps $\GG_m\to E$, $\lambda\mapsto (\lambda,u(\lambda))$
where $\lambda\mapsto u(\lambda)$ is a crossed homomorphism,
that is $u(\lambda\mu)=u(\lambda)+\lambda^{-1}u(\mu)$. Those
are all of the form $u(\lambda)=s(1-\lambda^{-1})$ for some
$s\in\alpha_p(R)$. In conclusion the functor of
$T$-linearizations of $P$ is isomorphic to $\alpha_p$ and
is not trivial, proving that $\sX^T\to\sX$ is not a monomorphism.
For more material on extensions with quotient of
multiplicative type, we refer to Demazure and
Gabriel~\cite{DG70}, Chap.~III, \S~6, n°~6.
\end{example}

\subsection{Geometric interpretation of group cohomology
in degrees 1 and 2} \label{G-torsors-and-G-gerbes}

In this section we consider a sheaf of groups $G$ and a sheaf
of $G$-modules $A$ over $S$ (that is, an abelian sheaf
endowed with an additive action of $G$) and we give the
interpretation of $H^1(G,A)$ and $H^2(G,A)$ in terms of
equivariant torsors and gerbes. Since this basically amounts
to reviewing the classical correspondences between geometric
objects and cohomological classes and proving that they are
$G$-equivariant, we sometimes omit some details.

We denote by $A^\circ$ the underlying abelian sheaf of $A$,
devoid of $G$-action.

\begin{definition}
An {\em $A$-torsor} is an $A^\circ$-torsor $P$ endowed with
an action of $G$ such that the
action morphism $A^\circ\times P\to P$ is $G$-equivariant.
\end{definition}

\begin{lemma} \label{prop-torsors-and-H1}
There is a canonical bijection between the set of isomorphism
classes of $A$-torsors over~$S$ and the cohomology group
$H^1(G,A)$.
\end{lemma}

In the proof below, starting from an $A$-torsor $P$
we will provide a construction of a canonical extension
$1\to A\to E\to \ZZ\to 1$ such that $P$ is the preimage of
$E\to\ZZ$ at $1\in\ZZ$. This is classical when $G=1$ and the
novelty here is to make sure that the procedure is
$G$-equivariant. Our construction is different from those of
\cite{Ol16}, 12.1.4 or \spref{02FQ}, having a more geometric
flavour. Moreover, if $A$ is a sheaf of possibly noncommutative
groups, the construction works equally well by working with
$(A,A)$-bitorsors instead of $A$-torsors, providing an extension
of the result to that case.

\begin{proof}
For an $A$-torsor $P$, recall that the opposite $A^\circ$-torsor
$P^{-1}$ is $P$ with the opposite $A^\circ$-action, that is
the action given by $a\cdot p=a^{-1}p$. The action of $G$ on $P$
commutes with the opposite $A^\circ$-action, turning $P ^{-1}$
into an $A$-torsor. Iterating the contracted product of torsors
denoted by a wedge, we define the {\em contracted powers} of $P$
as follows:
\[
P^{\wedge 0}:=A \ ; \qquad
P^{\wedge n+1}:=P^{\wedge n}\wedge P \mbox{ for all } n\ge 0 \ ;
\qquad
P^{\wedge n-1}:=P^{\wedge n}\wedge P^{-1} \mbox{ for all } n\le 0.
\]
When endowed with the diagonal action
$g\cdot (p_1\wedge\dots\wedge p_n)=gp_1\wedge\dots\wedge gp_n$,
where the $p_i$ are local sections of $P$ or $P^{-1}$, the
$A^\circ$-torsors $P^{\wedge n}$ become $A$-torsors for all
$n\in\ZZ$. We define:
\[
E:= \coprod_{n\in,\ZZ} P^{\wedge n}.
\]
The contracted product induces maps
$P^{\wedge m}\times P^{\wedge n}\to P^{\wedge m+n}$
(for all $m,n\in\ZZ$) which
assembled together endow $E$ with a group law such that the
map $f:E\to\ZZ$ mapping $P^{\wedge n}$ into $n$ is a
$G$-equivariant group homomorphism. Of course $\ker(f)=A$ and
$f^{-1}(1)=P$. The class of the extension
\[
1\too A\too E\too \ZZ\too 1
\]
defines the element of $\Ext^1_G(\ZZ,A)\simeq H^1(G,A)$
that completes the definition of the desired bijection.
The fact that the map is indeed bijective is easy and left to
the reader.
\end{proof}

For a reminder of the definition of a gerbe banded by an
abelian group, one can consult \cite{Ol16}, 12.2.2.

\begin{definition}
An {\em $A$-gerbe} is an $A^\circ$-gerbe $\sG$ endowed with
an action of $G$ such that for all sections $x\in\sG(T)$ over
some $S$-scheme $T$, the given isomorphism
$A^\circ_T\isomto\Aut_T(x)$ is $G_T$-equivariant.
\end{definition}

\begin{lemma} \label{prop-gerbes-and-H2}
There is a canonical bijection between the set of isomorphism
classes of $A$-gerbes over~$S$ and the cohomology group
$H^2(G,A)$.
\end{lemma}

Similarly as before, starting from an $A$-gerbe $\sG$ we
construct a length two extension $1\to A\to E\to F\to \ZZ\to 1$.
It would be very interesting to produce a {\em canonical}
extension. Since we do not know how (and fortunately we do not
need) to do this, we merely adapt the proof of \cite{Ol16}, 12.2.8.

\begin{proof}
Let $\sG$ be an $A$-gerbe. Choose an injection $i:A\to I$
into an injective sheaf of $G$-modules and let $K=I/A$ be
the quotient. Since $I$ is injective, the gerbe $i_*\sG$ is
neutral (see \cite{Ol16}, 12.2.9) hence there exists a section
$\eta:S\to\sG$. Let $\sP$ be the $K^\circ$-torsor of sections
of $\sG$ that induce $\eta$, defined as in the proof of
\cite{Ol16}, 12.2.8. Since the previous constructions are
$G$-equivariant, the torsor $\sP$ acquires a $G$-action making
it a $K$-torsor. Let $1\to K\to E\to\ZZ\to 1$ be the extension
attached to this torsor like in the proof
of~\ref{prop-torsors-and-H1}. We obtain a length 2 extension
of $G$-modules
\[
1\too A\too I\too E\too \ZZ\too 1
\]
whose class in $\Ext^2_G(\ZZ,A)\simeq H^2(G,A)$ defines
the desired bijection. Again, the verification that the
extension class does not depend on the choices of $i:A\to I$
and $\eta$,
and that the resulting map is bijective, being classical,
are left to the reader.
\end{proof}

\begin{definition}
A $G$-Picard stack is a Picard stack $\sP$ over $S$ endowed
with a $G$-action such that the addition morphism
$+:\sP\times\sP\to\sP$ is $G$-equivariant. We denote by
$\sP^\circ$ the underlying Picard stack, devoid of $G$-action.
\end{definition}

\begin{definition}
Let $\sP$ be a $G$-Picard stack. A $\sP$-torsor is an
$S$-stack $\sQ$ endowed with an action $\mu:\sP\times\sQ\to\sQ$
of $\sP$ such that
$(\mu,\pr_2):\sP\times\sQ\to\sQ\times\sQ$ is an isomorphism
(that is, the action is free and locally transitive).
\end{definition}

If $\sP$ is a $G$-Picard stack, the sheaf of isomorphism
classes $A$ and the sheaf of automorphisms of the neutral
object $e\in\sP(S)$ are sheaves of $G$-modules.

Endowed with the contracted product, the set of isomorphism
classes of $\sP^\circ$-torsors is a group denoted
$H^1(S,\sP^\circ)$, see~\cite{Bre90},~Prop.~6.2. Similarly
we can define a group of isomorphism classes of $\sP$-torsors
which we denote $H^1(G,\sP)$.
It is proved in \cite{Bro21}, Prop.~5.11 that
$H^2(S,A^\circ)=H^1(S,BA^\circ)$ and this group classifies
$A^\circ$-gerbes or $BA^\circ$-torsors (clearly
$(BA)^\circ=B(A^\circ)$). As we did before, one can follow
the constructions of the proof of {\em loc. cit.} and notice
that they are $G$-equivariant, thereby enhancing the previous
isomorphism to an isomorphism $H^2(G,A)=H^1(G,BA)$, both groups
classifying $A$-gerbes or $BA$-torsors. Since we will not need
this, we omit the details.

By assembling together torsors and gerbes, we can prove the
following triviality result for torsors under certain Picard
stacks which will be the key to the proof of
Theorem~\ref{thm:smoothness-fixed-points} below.

\begin{lemma} \label{lemma:torsors_under_Picard_stacks}
Let $\sP$ be a $G$-Picard stack over $S$. Let $P$ be the sheaf
of isomorphism classes and~$A$ the sheaf of automorphisms
of the neutral object $e\in\sP(S)$. If $H^1(G,P)=H^2(G,A)=0$
then $H^1(G,\sP)=0$, that is, all $\sP$-torsors over $S$
are trivial.
\end{lemma}

\begin{proof}
Let $\sQ\to S$ be a $\sP$-torsor, so we have an isomorphism:
\[
\sP\times\sQ\isomto\sQ\times\sQ.
\]
Passing to sheaves of isomorphism classes, we obtain:
\[
P\times Q\isomto Q\times Q,
\]
that is $Q\to S$ is a $P$-torsor. Since $H^1(G,P)=0$,
by Lemma~\ref{prop-torsors-and-H1}
this torsor has a section $q:S\to Q$. Let $\sG=q^*\sQ$,
a gerbe over $S$. The isomorphism
$\sP\times\sQ\isomto\sQ\times\sQ$ sends $(0,q)$ to $(q,q)$;
passing to inertia stacks in this isomorphism we obtain
\[
BA\times\sG\isomto\sG\times\sG,
\]
that is $\sG\to S$ is a $BA$-torsor. This means that $\sG$
is an $A$-gerbe; let us provide the easy verifications of
this. Let $T$ be an $S$-scheme
and $x\in\sG(T)$ a section. We thus have an isomorphism:
\[
f:A\times\Aut(x)\isomto \Aut(x)\times\Aut(x),
\quad f(a,u)=(u,a\cdot u).
\]
By computing the images of $(ab,\id_x)=(a,\id_x)(b,\id_x)$ in
two different ways, one finds that the map
$\iota_x:A\to\Aut(x)$, $a\mapsto a\cdot \id_x$ is a morphism
of groups. By using that $f$ is bijective as a sheaf map,
we find that the same is true for $\iota_x$. Finally, by using
that $f$ is $G$-equivariant we obtain the same conclusion
for~$\iota_x$. The collection of isomorphisms $\{\iota_x\}$
shows that $\sG$ is an $A$-gerbe. Since $H^2(G,A)=0$, by
Lemma~\ref{prop-gerbes-and-H2} this gerbe has a section
$\alpha:S\to\sG$. Using homogeneity we have
$\sQ\isomto Q\times\sG$ which the section
$(q,\alpha):S\to Q\times\sG$ trivializes and finally $\sQ$
is trivial.
\end{proof}

\subsection{Smoothness}

In this subsection, we study the smoothness of
fixed point stacks. For fixed point schemes, a useful
reference is \cite{SGA3.2}, Expos\'e XII, \S~9 (the reader
should be careful however that in Prop.~9.2 of {\em loc. cit.}
the assumption that $X$ is separated over $S$ is missing).

\begin{notitle}{The equivariant cotangent complex}
Let $\sX\to S$ be a smooth algebraic stack. We want to recall
the elementary description of the cotangent complex in this
context; since we will have
to handle stacks endowed with a group action, it is appropriate
to work with sheaves on the equivariant site. The reader is
assumed to be familiar with basics on equivariant
quasi-coherent sheaves on schemes, like for instance in
\cite{AOV08}, \S~2.1.

Let $G\to S$ be a flat, locally finitely presented group
algebraic space acting on $\sX\to S$.
The {\em equivariant lisse-\'etale site} is the site
$\Liset{}\!^G(\sX)$ whose underlying category is
the category of smooth $G$-schemes $U\to \sX$ (that is
$G$-schemes $U\to S$ with a smooth equivariant morphism
$U\to \sX$), and whose covering families $\{U_i\to U\}_{i\in I}$
are families of \'etale $G$-equivariant morphisms such that
$\amalg_{i\in I} U_i\to U$ is surjective. Particular objects
of this site can be obtained as pullbacks
$U\defeq V\times_{\sX\!/G}\sX$ of objects $V\to \sX\!/G$
of the ordinary, non-equivariant lisse-\'etale site of the
quotient stack. In particular, we see that $\sX$ has
equivariant smooth atlases.

We define the
{\em equivariant cotangent complex $\LL_{\sX\!/S}$} of
$(G,\sX)$ as an object of the derived category of bounded
complexes of $G$-quasi-coherent modules.
Let $f:U\to \sX$ be an object in $\Liset{}\!^G(\sX)$. Choose
a $G$-atlas $V\to \sX$ and write $f':U\times_{\sX}V\to V$
the pullback of $f$. Then the sheaf
${\Omega^1_{U/\sX}}{}_{|U\times_{\sX}V}\defeq
\Omega^1_{U\times_{\sX}V/V}$
descends along $U\times_{\sX}V\to U$ to a $G$-quasi-coherent
$\cO_U$-module which we denote $\Omega^1_{U/\sX}$.
For each object $f:U\to X$ in $\Liset{}\!^G(\sX)$, we define
a length two complex with sheaves placed in degrees 0 and 1:
\[
\LL_{\sX\!/S}{}_{|U}\defeq\left[\Omega^1_{U/S}\too\Omega^1_{U/\sX}\right].
\]
If $f:V\to U$ is a morphism in $\Liset{}\!^G(\sX)$, there
is a commutative diagram
\[
\begin{tikzcd}
f^*\Omega^1_{U/S} \ar[r] \ar[d] & \Omega^1_{V/S} \ar[d] \\
f^*\Omega^1_{U/\sX} \ar[r] & \Omega^1_{V/\sX}
\end{tikzcd}
\]
which induces a quasi-isomorphism
\[
\theta_f:\left[f^*\Omega^1_{U/S}\to f^*\Omega^1_{U/\sX}\right]
\too
\left[\Omega^1_{V/S}\to \Omega^1_{V/\sX}\right].
\]
Moreover, for $W/V/U$
in $G\mbox{-}\Liset(\sX)$ these quasi-isomorphisms satisfy the
cocycle condition. The equivariant cotangent complex is the
complex defined by the data $(\LL_{\sX\!/S}{}_{|U},\theta_f)$.
\end{notitle}

\begin{notitle}{Deformations of sections of $\sX\to S$}
In what follows we work on both the small \'etale site
$S_{\et}$ and the big fppf site $S_{\fppf}$. We denote by
$\eps:S_{\fppf}\to S_{\et}$ the canonical morphism.
For basics on Picard stacks and their torsors, we refer
to Deligne~\cite{SGA4.3}, Exp.~XVIII, \S~1.4,
Breen~\cite{Bre90}, section~6 and Brochard~\cite{Bro21},
sections~2 and 5.

Let $I$ be a quasi-coherent $\cO_S$-module.
Let $\Thick(S,I)$ be the category of {\em thickenings of $S$
by~$I$}, which by definition are pairs $(S\into S',u)$ composed
of a closed immersion of schemes defined by a square-zero
ideal, and an isomorphism $I\simeq\ker(\cO_{S'}\to\cO_S)$
which most often is omitted from the notation. There is a
stack $\sThick(S,I)$ on $S_{\et}$,
whose fibre category over $U\to S$ is $\Thick(U,I_{|U})$.
(Here we are forced to work on the small \'etale site
in order to guarantee existence of pullbacks: only for
\'etale $T\to S$ does the thickening $S\into S'$ lift
uniquely to a thickening $T\into T'$.)
This is endowed with the structure of a Picard stack whose
neutral object is the thickening $S\into S[I]$ where
$S[I]=\Spec(\cO_S\oplus I)$ with $I^2=0$.

The tangent stack of $\sX$ relative to $I$ is the stack
$\sT_{\sX\!/S}(I)\defeq\sHom(S[I],\sX)$ on $S_{\fppf}$
whose points are the morphisms $S[I]\to\sX$.
It comes endowed with a morphism $\sT_{\sX\!/S}(I)\to \sX$
induced by the immersion $S\into S[I]$. If $x:S\to\sX$
is a section, the pullback $x^*\sT_{\sX\!/S}(I)$
is the stack of morphisms extending $x$.
Since $\sX\to S$ is smooth, the usual computation shows
that $x^*\sT_{\sX\!/S}(I)$ is canonically and
equivariantly isomorphic to the stack associated as in
Deligne's expos\'e \cite{SGA4.3}, Exp.~XVIII, \S~1.4
to the length two complex
$\tau_{\le 0}R\Hom(x^*\LL_{\sX\!/S},I)$.

Let $\Exal_{\sX}(S,I)$ be the category whose objects
are the pairs composed of a thickening $S\into S'$ of
$S$ by $I$, and a morphism $x':S'\to\sX$ extending $x$
(this includes the datum of a 2-isomorphism
$u:x'_{|S}\simeq x$). There is a Picard stack
$\sExal_{\sX}(S,I)$ on $S_{\et}$, whose fibre
category over $U\to S$ is $\Exal_{\sX}(U,I_{|U})$.
Moreover, this sits in an exact sequence of Picard stacks:
\begin{equation}\label{eq-SES-Picard-stacks}
0 \too x^*\sT_{\sX\!/S}(I) \too \sExal_{\sX}(S,I)
\too \sThick(S,I) \too 0.
\end{equation}
Here, exactness on the right is guaranteed by the
smoothness of $\sX\to S$. It follows that the fibre of
$\sExal_{\sX}(S,I) \to \sThick(S,I)$
above a given thickening $S\into S'$ is a torsor under
$x^*\sT_{\sX\!/S}(I)$.
\end{notitle}

\begin{notitle}{Group action}
Now assume that $x:S\to \sX$ is fixed by $G$, which means
that it is given with a collection of isomorphisms
$\{\alpha_g:gx\simeq x\}_{g\in G}$. In this case $G$ acts
on $\sExal_{\sX}(S,I)$ as follows:
\[
g\cdot (x',u)=(g\circ x',\alpha_g\circ gu).
\]
The group $G$ also acts on $x^*\sT_{\sX\!/S}(I)$ by the
same formula; it acts trivially on $\sThick(S,I)$ and the
exact sequence (\ref{eq-SES-Picard-stacks}) is equivariant.

From this, it is natural to approach
Theorem~\ref{thm:smoothness-fixed-points} by descending the
sequence (\ref{eq-SES-Picard-stacks}) to an exact sequence
on the small \'etale site of $BG$ and proving smoothness via
triviality of a certain $x^*\sT_{\sX\!/S}(I)$-torsor on $BG$.
This indeed works well; however, since proper foundations
for Picard stacks over a
stacky site such as $(BG)_{\et}$ are lacking, we prefer to work
with equivariant objects using the material developed in
Subsection~\ref{G-torsors-and-G-gerbes}.
\end{notitle}

\begin{notitle}{Linearly reductive group schemes}
We can now introduce linearly reductive group schemes and
prove the statement of
smoothness in Theorem~\ref{th-2}. We use the notion
of linear reductivity given by Alper~\cite{Alp13},
Def.~12.1; see also Brion \cite{Bri21}.
For recent results concerning {\em affine} linearly
reductive group schemes the reader may look at
Alper, Hall and Rydh's article~\cite{AHR21}, Section~19.
\end{notitle}

\begin{definition} \label{defi:lin_red_gp}
A flat, finitely presented, separated group scheme
$G\to S$ is called {\em linearly reductive} if the
functor $\QCoh^G(S)\to\QCoh(S)$, $\cF\mapsto \cF^G$ is exact.
\end{definition}

Our interest for linearly reductive group schemes is that
if $S$ is affine, the higher Hochschild cohomology of
quasi-coherent $G$-$\cO_S$-modules vanishes, as follows
from the definition:
\[
H^i(G,\cF)=0 \mbox{ for all }
\cF\in \QCoh^G(S) \mbox{ and } i\ge 1.
\]
The class of linearly reductive group schemes is stable by
base change, faithfully flat descent (\cite{Alp13},
Prop.~12.8), subgroups with affine quotient (Matsushima's Theorem,
\cite{Alp13}, Th.~12.15),
and extensions (\cite{Alp13}, Prop.~12.17). It contains:
\begin{trivlist}
\itemm{1} group schemes of multiplicative type, by \cite{SGA3.2},
Exp.~IX, Th.~3.1;
\itemm{2} finite locally free group schemes of order
invertible
on $S$, by the existence of explicit avering operators;
\itemm{3} abelian schemes, by \cite{Alp13}, Ex.~12.4;
\itemm{4} reductive group schemes, if $S$ is a $\QQ$-scheme,
by the following results of \cite{Alp14}:
such a group scheme is geometrically reductive (Th.~9.7.5),
hence $BG\to S$ is adequately affine (Def.~9.1.1)
and cohomologically affine (Lem.~4.1.6) which is the definition
of  linearly reductive (Rem.~9.1.3).
\end{trivlist}

\begin{theorem} \label{thm:smoothness-fixed-points}
Let $S$ be a scheme and $\sX\to S$ an algebraic stack
with affine, finitely presented diagonal. Let $G\to S$ be a
linearly reductive group scheme. If $\sX\to S$ is smooth,
the fixed point stack $\sX^G\to S$ is smooth.
\end{theorem}

\begin{proof}
By Theorem~\ref{th-fixed-pt-stack} the stack $\sX^G\to S$
is algebraic and locally
of finite presentation, hence it is enough to prove that
it is formally smooth. Let $x:T\to \sX^G$ be a point with
values in some $S$-scheme $T$ which is affine. After
base-changing $\sX$ along $T\to S$ and renaming, we can assume
that $T=S$ in what follows. We have the sequence of Picard
stacks over $S$:
\[
0 \too x^*\sT_{\sX\!/S}(I) \too \sExal_{\sX}(S,I)
\stackrel{\pi}{\too} \sThick(S,I) \too 0.
\]
Since $\sX\to S$ is smooth, this sequence is exact. Moreover,
as explained before, these stacks are naturally endowed with
$G$-actions (the $G$-action on the stack of thickenings is
trivial) and the sequence is equivariant.
Let $\iota:S\into S'$ be a thickening; we have to prove that
$x$ has a lifting $x':S'\to \sX^G$. The
category of liftings of $x$ along $\iota$ is the fibre category
$\pi^{-1}(\iota)\subset \sExal_{\sX}(S,I)$. This is a
torsor under the $G$-Picard stack $x^*\sT_{\sX\!/S}(I)$,
whose sheaves of isomorphism classes $P$ and neutral
automorphisms $A$ are quasi-coherent. Because $G\to S$ is
linearly reductive, we have $H^1(G,P)=H^2(G,A)=0$.
By lemma~\ref{lemma:torsors_under_Picard_stacks}, this implies
that this torsor is trivial. In other words, it has a section
$x':S\to \pi^{-1}(\iota)$ which gives the desired lifting for~$x$.
\end{proof}

\subsection{Applications}

We conclude by giving several applications of
Theorem~\ref{thm:smoothness-fixed-points}. The first application
is to the flatness or smoothness properties of the space of
homomorphisms $\Hom(G,H)$. It relies on the following
well-known fact, a proof of which we provide for the
convenience of the reader.

\begin{lemma} \label{homs_between_classifying_stacks}
Let $G,H$ be sheaves of groups over a base scheme~$S$
(for some topology). Then there is an isomorphism of stacks
\[
[\Hom(G,H)/H]\isomto \sHom(BG,BH)
\]
where the quotient is taken for the action of $H$ on
$\Hom(G,H)$ by conjugation on the target. 
\end{lemma}

\begin{proof}
We define maps in both directions. A section
of the stack $[\Hom(G,H)/H]$ is a pair composed of an $H$-torsor
$S'\to S$ and an $H$-equivariant map $f:S'\to \Hom(G,H)$. With
these data we define a map $\Phi:BG\to BH$ as follows. Let
$\varphi:G_{S'}\to H_{S'}$ be the group homomorphism determined
by $f$. To a $G$-torsor $E$ we attach the $H_{S'}$-torsor
$F_{S'}\defeq E_{S'}\wedge^\varphi H_{S'}$.
To say that $f$ is equivariant is to say that for all local sections
$h\in H$, the pullback of $\varphi$ along $h:S'\to S'$ is equal to
$h^*\varphi=c_h\circ \varphi$ where $c_h:H\to H$ is conjugation.
Therefore, for all local sections $h\in H$ we have, canonically:
\[
h^*F_{S'}=E_{S'}\wedge^{c_h\varphi} H_{S'}.
\]
This implies that the isomorphism of $H_{S'}$-torsors
\[
(\id_{E_{S'}},c_h):E_{S'}\wedge^\varphi H_{S'}\too
E_{S'}\wedge^{c_h\varphi} H_{S'}
\]
is an isomorphism $F_{S'}\isomto h^*F_{S'}$. This gives descent
data for $F_{S'}$ with respect to $S'\to S$, and we call $\Phi(E)$
the descent $F\to S$.
Conversely let $\Phi:BG\to BH$ be a morphism of stacks and let
$S'\in BH$ be the image of the trivial torsor $G$ by $\Phi$. After the
pullback $S'\to S$ the torsor $F$ becomes trivial and the map
$\Aut_S(G)\to \Aut_S(\Phi(G))$ becomes a morphism of groups:
\[
\varphi':G_{S'}=\Aut_{S'}(G_{S'}) \too \Aut_{S'}(F_{S'})=\Aut_{S'}(H_{S'})=H_{S'}.
\]
This amounts to a morphism $u:S'\to \Hom(G,H)$. For each local
section $h\in H$, we have $S'$-group schemes $a:G_{S'}\to S'$ and
$b:H_{S'}\to S'$ and the pullbacks $h^*G_{S'}$, $h^*G_{S'}$ are
isomorphic to $G_{S'},H_{S'}$ with structure maps $h^{-1}\circ a$
and $h^{-1}\circ b$. This shows that $h^*\varphi'$ is
$c_h\circ\varphi'$, hence $u$ is $H$-equivariant. It provides
a point of the quotient $[\Hom(G,H)/H]$. The two maps so
described are inverse to each other.
\end{proof}

\begin{corollary} \label{cor:smoothness-Hom-functor}
Let $G\to S$ be a linearly reductive $S$-group scheme.
Let $H$ be a flat, affine, finitely presented $S$-group scheme.
Then the stack $\sHom(BG,BH)\to S$ is algebraic and smooth.
In particular,
\begin{trivlist}
\itemm{i} $\Hom(G,H)\to S$ is flat and locally complete
intersection,
\itemm{ii} $\Hom(G,H)\to S$ is smooth if moreover $H\to S$
is smooth.
\end{trivlist}
\end{corollary}

\begin{proof}
The stack $\sX=BH$ has affine, finitely presented diagonal
and it is smooth because its natural atlas $S\to BH$ has smooth
source (see \spref{0DLS}). Letting $G$ act trivially on it,
we obtain $\sX^G=\sHom(S,BH)^G=\sHom(BG,BH)$ which is smooth
by Theorem~\ref{thm:smoothness-fixed-points}. The group
scheme $H\to S$ is locally complete intersection in case~(i)
and smooth in case~(ii); since $\Hom(G,H)\to\sHom(BG,BH)$
is an $H$-torsor by Lemma~\ref{homs_between_classifying_stacks},
the announced properties are deduced.
\end{proof}

The reader can find other examples of $\sHom$ stacks
in~\ref{examples-transitivity}(1) below. 

\begin{remark}
({\em Smoothness in case {\rm (ii)} by deformation theory.})
Because $\Hom(G,H)\to S$ is locally of finite presentation,
in order to prove that it is smooth when $H$ is smooth, it is
enough to verify the infinitesimal lifting criterion. Given a
nilimmersion of affine schemes $T\into T'$ over $S$, and a
$T$-morphism $f:G_T\to H_T$, we seek to lift it to
a $T'$-morphism $f':G_{T'}\to H_{T'}$. Changing
notation, we can assume that $T=S$, $T'=S'$ which are
affine schemes. Let $I=\ker(\cO_{S'}\to\cO_S)$ be the
square-zero kernel. By Illusie~\cite{Il72}, Chap.~VII,
Th.~3.3.1 the obstruction to lifting $f$ lives in
$\HH^2(BG/S,f^*\underline{\ell}{}_H^\smallvee\otimes I)$
where $\underline{\ell}{}_H \in D(BH)$ is the equivariant
co-Lie complex of $H\to S$; here the operations on
complexes $f^*$, $(-)^\smallvee$, $\otimes$ are
understood in the derived sense. Therefore it is enough
to prove that this cohomology group is zero. Let us write
once for all
\[
K:=f^*\underline{\ell}{}_H^\smallvee\otimes I.
\]
Note that $\HH^i(BG/S,K)$ denotes {\em relative} cohomology
with respect to the map $S\to BG$, as in \cite{Il71},
Chap.~III, \S~4. This is related to ordinary cohomology
$\HH^i(BG,K)$ via a long exact sequence of which we write
the part which is useful for our calculation:
\[
\dots \too \HH^1(S,K) \too \HH^2(BG/S,K) \too
\HH^2(BG,K) \too \HH^2(S,K) \too \dots
\]

We claim that $\HH^2(BG,K)=0$. To compute this group, we
use the second hypercohomology spectral sequence:
\[
{}^{\mathrm{II}}\!E_2^{i,j}
=H^i(BG,\cH^j(K))
\Longrightarrow
\HH^{i+j}(BG,K).
\]
We know that quasi-coherent cohomology on $BG$ coincides with
group cohomology of $G$. Hence, by the assumption
on $G$ we have $H^i(BG,\sF)=H^i(G,\sF)=0$ for $i\ge 1$
for all quasi-coherent sheaves $\sF$ on~$BG$. It
follows that the spectral sequence collapses at $E_2$,
giving isomorphisms
\[
\HH^n(BG,K) \simeq
H^0(BG,\cH^n(K))
=\Gamma^G(\cH^n(K))
\]
for all $n$, where $\Gamma^G(-)$ denotes $G$-invariant
global sections. We claim that the sheaf
$\cH^n(K)$ vanishes
for all $n\ge 2$. The complex
$f^*\underline{\ell}{}_H^\smallvee$
has perfect amplitude in $[0,1]$ (\cite{Il72}, Chap.~VII,
\S~3.1); thus replacing~$I$ by a flat resolution
$\dots\to I_{-2}\to I_{-1}\to I_0$ and computing the total
complex of $f^*\underline{\ell}{}_H^\smallvee\otimes I$ we find
that it has no cohomology in degrees $\ge 2$, as claimed.

To compute $\HH^1(S,K)$ we can proceed similarly. Since $S$
is affine, the quasi-coherent sheaves $\cH^j(K)$ have no
higher cohomology. It follows that the hypercohomology
spectral sequence degenerates at $E_2$, giving isomorphisms
$\HH^n(S,K) \simeq H^0(S,\cH^n(K))$. In case
\ref{cor:smoothness-Hom-functor}(i) the complex
$f^*\underline{\ell}{}_H^\smallvee$ is locally represented
by a two-term complex of quasi-coherent sheaves $[F_0\to F_1]$;
then the total complex has a term $F_1\otimes I_0$ in degree~1
and we can not conclude to the vanishing of $\HH^1(S,K)$.
In case \ref{cor:smoothness-Hom-functor}(ii) however, the
complex~$\underline{\ell}{}_H$ is quasi-isomorphic to the sheaf
of invariant differentials $\omega^1_{H/S}$ viewed as a complex
concentrated in degree~0 and $\HH^1(S,K)$ vanishes. In this case
$\HH^2(BG/S,K)=0$ and we are done.

Should one want to study potential cases of smoothness of
$\Hom(G,H)$ in case~(i), a natural approach would be to
undertake a more detailed analysis of $\HH^1(S,K)$.
\end{remark}

\begin{examples}
({\em Non-smooth examples.})
Here are various examples of schemes $\Hom(G,H)$; the cases
in (2) and (3) were communicated by Michel Brion and Angelo Vistoli.
\begin{trivlist}
\itemm{1} When $H$ is not flat the scheme $\Hom(G,H)$ can exhibit
all kind of behaviour: it can be smooth, or on the contrary
not even flat. To give examples, let $R$ be a discrete valuation
ring with residue field $k$. If~$M$ is a $k$-group scheme with
affine identity component, we let $M^\natural$ be the scheme
obtained by gluing
the trivial $R$-group scheme $\{1\}_R$ with $M$ along the
unit section $\{1\}_k$ of the special fibre. This is an
$R$-group scheme with affine identity component, which is
non-flat when $M\ne \{1\}_k$. Then, one has:
\begin{trivlist}
\itemm{i}
$\Hom(\GG_{m,R},(\GG_{a,k})^\natural)
\into \Hom(\GG_{m,R},\GG_{a,R})=1$
hence $\Hom(\GG_{m,R},(\GG_{a,k})^\natural)=1$
which is smooth;
\itemm{ii}
$\Hom(\GG_{m,R},(\GG_{m,k})^\natural)=
(\underline\ZZ_k)^\natural$ by a direct computation;
this is not flat.
\end{trivlist}
\itemm{2} Let $S=\Spec(k)$ where $k$ is a separably closed,
non algebraically closed field of characteristic $p$. Take
$G=\GG_m$ and $H$ a nontrivial extension of $\GG_m$ by $\alpha_p$;
such extensions are described in \cite{SGA3.2}, Exp.~XVII,
Exemple 5.9.c). If $\Hom(G,H)$ is representable by a smooth scheme
$T$, then $T(k)$ is a point because each homomorphism $G\to H$
is trivial (for otherwise its image would split the extension).
Since the extension is split on the algebraic closure $\bar{k}$,
we have $T(\bar{k})=\Hom(\GG_{m,\bar{k}},\GG_{m,\bar{k}})=\ZZ$.
But since $T$ is smooth $T(k)$ is dense in $T$, a contradiction.
\itemm{3} Here is an example where $G$ and $H$ are both finite and
linearly reductive, and the field $k$ is algebraically closed
of characteristic $p$. Suppose that $\Delta$ is a finite connected
diagonalizable $k$-group scheme, $H$ is a nontrivial group of
automorphisms of $\Delta$, and consider the semidirect product
$G := \Delta \rtimes H$. Then $\Hom(G,G)$ is representable and
finite, and contains $\Aut(G)$ as an open subscheme. It is shown
in \cite{AOV08}, Lemma 2.19 that the connected component of the
identity in $\Aut(G)$ is $\Delta/\Delta^H$, which is not reduced,
so $\Hom(G,G)$ can not be smooth.
\end{trivlist}
\end{examples}

We now give an application to functors of subgroups.

\begin{corollary} \label{coro-sub-mult-smooth}
Let $H\to S$ be an affine, flat group scheme of finite
presentation. Then:
\begin{trivlist}
\itemn{1} The functor of subgroups of multiplicative type
$\Sub_{\mult}(H)$ is flat and locally complete intersection
over $S$, and it is smooth if $H$ is.
\itemn{2} If $S$ is of characteristic~0 then $\Sub_{\red}(H)$
is smooth over $S$.
\end{trivlist}
\end{corollary}

\begin{proof}
(1) Recall that
$\Sub_{\mult}(H)={\coprod}_{\type} \Sub_{\type}(H)$
is a sum indexed by the types
$\type=[(M,M^*,\varnothing,\varnothing)]$.
Since $G(\type)$ is of multiplicative type, it is linearly
reductive hence $\Hom(G(\type),H)$ is flat and locally complete
intersection (resp. smooth if $H$ is)
by~\ref{cor:smoothness-Hom-functor}. Since
$\Mono(G(\type),H)$ is open in the latter by
Lemma~\ref{lemma:Mono_reductive_separated}, it is flat and
locally complete intersection also. Finally, remember that in
the proof of Theorem~\ref{theorem:subgroups_reductive_affine}
we expressed $\Sub_{\type}(H)$ as the quotient of
$\Mono(G(\type),H)$ by the smooth group scheme
$\Aut(G(\type))$ acting freely; the result follows.

\smallskip

\noindent (2)
If $S$ is of characteristic~0, for an arbitrary type $\type$
the reductive group schemes $G(\type)$ are linearly reductive
and the flat group scheme $H$ is smooth. Therefore
the same arguments apply.
\end{proof}

The final application extends~\ref{cor:smoothness-Hom-functor}
as well as \cite{Ro05}, Cor.~3.11.

\begin{corollary} \label{coro:stack_of_equivariant_objects}
{\em (Stacks of equivariant objects.)}
Let $\sX\to S$ be an algebraic stack with affine, finitely
presented diagonal. Let $G\to S$ be a group space admitting
a finite composition series whose factors are either
reductive or proper, flat, finitely presented.
\begin{trivlist}
\itemn{1} The stack $\sX[G]$ of pairs $(x,\alpha)$ comprising
an object of $\sX$ and an action $\alpha:G\to\Aut(x)$ is
algebraic. Moreover $\sX^G\to \sX$ is representable by algebraic
spaces, separated and locally of finite presentation.
\itemn{2} The substack $\sX\{G\}$ composed of
pairs such that the action $\alpha$ is faithful is open.
\itemn{3} If $\sX\to S$ is smooth and $G$ is linearly
reductive, then $\sX[G]\to S$ and $\sX\{G\}\to S$ are smooth.
\end{trivlist}
\end{corollary}

\begin{proof}
Letting $G$ act trivially on $\sX$, we see that the fixed
point stack is exactly $\sX[G]$. By
Theorem~\ref{th-fixed-pt-stack},
this is algebraic and the substack $\sX\{G\}$ is open.
Finally the smoothness in (3) follows from
Theorem~\ref{thm:smoothness-fixed-points}.
\end{proof}

\subsection{Failure of transitivity (erratum to~\cite{Ro05})}
\label{failure}

In \cite{Ro05}, Rem.~2.4 it is asserted that if $N\subset G$
is a normal subgroup scheme, then:
\begin{trivlist}
\itemm{i} the group $G/N$ acts on $\sX^N$ and
we have an isomorphism of stacks $\sX^G\isomto (\sX^N)^{G/N}$;
\itemm{ii} the group $G/N$ acts on $\sX/N$ and
we have an isomorphism of stacks $(\sX/N)/(G/N)\isomto \sX/G$.
\end{trivlist}
In this subsection we wish to correct this statement: in
fact, surprisingly point (i) is {\em incorrect}
(Lemma~\ref{lemma:counterexample_fixed_pts})
while point (ii) is {\em correct}
(Proposition~\ref{prop:transitivity_quotients}). To the
author's knowledge the erroneous statement (i) is not used
anywhere, but (ii) is used in the papers \cite{LMM14},
\cite{Sch17}, \cite{Sch18}, \cite{AI19}.

\bigskip

To understand what happens, recall that an action
$\mu:G\times \sX\to \sX$ is called {\em strictly trivial}
if $\mu$ is equal to the second projection, which we write
$\mu=\triv$. The action is called {\em weakly trivial}
if there exists a $G$-isomorphism
$(f,\sigma):(\sX,\mu)\isomto (\sX,\triv)$ with $f=\id_\sX$;
here $\sigma_g^x$ is an isomorphism $g.f(x)\isomto f(g.x)$,
as in \cite{Ro05}, Def.~2.1. By the 2-universal
properties of quotients and fixed points, the stacks $\sX^N$
and $\sX/N$ come equipped with actions of $G$, the
restriction to $N$ of which are {\em weakly} trivial.
However, in order to induce an action of $G/N$ one needs
to find $G$-equivariant models of $\sX^N$ and $\sX/N$ on which
the action of $N$ is {\em strictly} trivial, and it is not
clear if this is possible at all.

We first provide a counterexample to (i). The example
highlights the fact that $G$-fixed point stacks retain
information on the extension structure of $G$ which is
not captured by $(\sX^N)^{G/N}$.

\begin{lemma} \label{lemma:counterexample_fixed_pts}
Let $G,N,A$ be $S$-group schemes of multiplicative type
with $G\to S$ fibrewise connected and $N\subset G$ a normal
subgroup.
Let $\sX=BA$ be the classifying stack of $A$, endowed with
the trivial action of $G$. Then:
\begin{trivlist}
\itemn{1} We have an isomorphism of stacks
$\sX^N=BA\times\Hom(N,A)$ such that the canonical map
$\sX^N\to \sX$ is the first projection.
\itemn{2} Each action of $G$ on $\sX^N$ making $\sX^N\to \sX$
equivariant is isomorphic to the trivial action.
\itemn{3} Letting $G/N$ act trivially on $\sX^N$,
we have an isomorphism of stacks
\[
(\sX^N)^{G/N} = BA\times\Hom(G/N,A)\times\Hom(N,A)
\]
and the canonical map $\sX^G\to (\sX^N)^{G/N}$ is given
by $(E,\alpha)\mapsto (E,0,\alpha_{|N})$. In particular,
this is not an isomorphism if $\Hom(G/N,A)\ne 0$.
\end{trivlist}
\end{lemma}

\begin{proof}
(1) Since $N$ acts trivially, a section of $(BA)^N$ over
$S$ is a pair $(E,\{\alpha_n\}_{n\in N(S)})$ composed of
an $A$-torsor $E\to S$ and a collection of isomorphisms
$\alpha_n:E\to E$ satisfying
the cocycle condition $\alpha_{nm}=\alpha_n\circ\alpha_m$
for all $n,m\in N(S)$. Since $\Aut(E)=A$, this boils down
to a pair $(E,\alpha)$ where $\alpha:N\to A$ is a morphism
of groups.

\smallskip

\noindent (2) Write $f:\sX^N\to \sX$ for the map
$(E,\alpha)\mapsto E$.
Assume given a $G$-action $(\sX^N,\mu)$ such that $f$
extends to a $G$-equivariant morphism $(f,\sigma):\sX^N\to \sX$.
For each $S$-scheme $T$ and points $(E,\alpha)\in \sX^N(T)$,
$g\in G(T)$ write the image $g\cdot (E,\alpha)$ as
$(\mathfrak{E}(g,E,\alpha),\mathfrak{a}(g,E,\alpha))$
where $\mathfrak{E}(g,E,\alpha)$ is an $A$-torsor and
$\mathfrak{a}(g,E,\alpha)$ is an $N$-linearization.
Now let $(E,\alpha)$ and $g$ be fixed and write
$\mathfrak{E}=\mathfrak{E}(g,E,\alpha)$,
$\mathfrak{a}=\mathfrak{a}(g,E,\alpha)$ for
brevity. We have an isomorphism
$\sigma\defeq\sigma_g^{E,\alpha}:E\to \mathfrak{E}$.
Define $\alpha'$ by $\alpha'(n):E\to E$, $\alpha'(n)=
\sigma^{-1}\circ \mathfrak{a}(n)\circ\sigma$. By setting
$g\star (E,\alpha)=(E,\alpha')$ we define a new $G$-stack
$(\sX^N,\mu')$ for which $f$ is strictly (and not just
weakly) invariant, together with a $G$-isomorphism
$(\sX^N,\mu')\isomto (\sX^N,\mu)$ provided by the
$\sigma_g^{E,\alpha}$. Moreover the dependence in $g$ for
$\mathfrak{a}'=\mathfrak{a}'(g,E,\alpha)$ is morphic,
that is $g\mapsto \mathfrak{a}'(g,E,\alpha)$ is an
action of $G$ on $\Hom(N,A)$. Since $G$ is connected and
$\Hom(N,A)$ is \'etale, this action is trivial.

\smallskip

\noindent (3) Reasoning as in (1) we find that
$(\sX^N)^{G/N} = BA\times\Hom(G/N,A)\times\Hom(N,A)$.
The expression for the map $\sX^G\to (\sX^N)^{G/N}$ is
dictated by the universal properties.
\end{proof}

\begin{examples} \label{examples-transitivity}
Here are examples where the map
$\sX^G\to (\sX^N)^{G/N}$ {\em is} an isomorphism.
\begin{trivlist}
\itemn{1} ($\sHom$ stacks) Assume that $\sX=\sHom(Y,Z)$ where
$Y,Z$ are algebraic stacks, and we are given an action of
$G$ on $Y$, inducing an action on $\sX$. By taking for $Z$
the stack of vector bundles (or coherent modules, or curves,
etc) we obtain for $\sX^G$ the stack of equivariant bundles
(or coherent modules, etc) on~$Y$. We claim that if
$N\subset G$ is a normal subgroup
then the map $\sX^G\to (\sX^N)^{G/N}$ is an isomorphism.
Indeed, it follows from the 2-universal property of quotient
stacks and Proposition~\ref{prop:transitivity_quotients}
below that we have canonical isomorphisms:
\[
\sX^G=\sHom(Y,Z)^G\simeq\sHom(Y/G,Z)\simeq\sHom((Y/N)/(G/N),Z)
\simeq\sHom(Y/N,Z)^{G/N}\simeq (\sX^N)^{G/N}.
\]
\itemn{2} (Product groups)
Assume that $G=H\times K$ is a product
group acting on $\sX$. Identify $G/H$ with $K$. Then
the canonical map $\sX^G\isomto (\sX^H)^K$ is an
isomorphism. The objects of the stacks on both sides can be
identified with collections
$(x,\{\alpha_{h}\}_{h\in H},\{\beta_{k}\}_{k\in K})$
where the isomorphisms $\alpha_h:x\to h^{-1}x$ and
$\beta_k:x\to k^{-1}x$ commute with each other.
\end{trivlist}
\end{examples}

In view of this, point (ii) may now seem surprising. We now
give the proof. The main idea is to strictify the action
by systematically embedding an $N$-torsor $E$ into
the induced $G$-torsor $\Ind_N^G(E)$.

\begin{proposition} \label{prop:transitivity_quotients}
Let $G\to S$ be a group scheme and $N\subset G$ a normal
subgroup scheme, both flat and locally of finite presentation.
Let $\sX\to S$ be an algebraic stack with an action of $G$.
Then there is an action of $G/N$ on the quotient stack $\sX/N$
such that the morphism $\sX/N\to \sX/G$ is invariant and induces
an isomorphism $(\sX/N)/(G/N)\isomto \sX/G$.
\end{proposition}

Note that the algebraicity statement \cite{Ro05}, Th.~4.1
assumes too strong assumptions on $G$ and $N$ (namely they
are required to be separated and of finite presentation)
than is necessary for the proof of {\em loc. cit.}
to go through.

\begin{proof}
Our main task is to find a $G$-equivariant model $\sX/N\simeq Y$
such that the $N$-action on $Y$ is strictly trivial, so that
there is an induced action of $G/N$. To this aim,
recall that the points of $\sX/N$ with values in an $S$-scheme
$T$ are the pairs composed of an $N$-torsor $E\to T$ and an
$N$-equivariant map $a:E\to \sX$. Let $Y$ be the $S$-stack whose
$T$-points are the triples $(F\to T,E'\subset F,b:F\to \sX)$ with
\begin{itemize}
\item a $G$-torsor $F\to T$;
\item a subspace $E'\subset F$ which is $N$-stable and an
$N$-torsor over $S$;
\item a $G$-equivariant map $b:F\to \sX$.
\end{itemize}
There is a morphism $\lambda:\sX/N\too Y$ which sends
$(E\to T,E\to \sX)$ to the triple given by
\begin{itemize}
\item $F=\Ind_N^G(E)=G\times^NE$ is the induction of
the $N$-torsor $E$ to $G$, that is $F=(G\times E)/N$
where~$N$ acts by $n(g,e)=(gn^{-1},ne)$;
\item $E'$ is the image of the monomorphism $E\to F$,
$e\mapsto (1,e)$; this may alternatively be seen as the
preimage of 1 under the map $F\to G/N$, $(g,e)\mapsto (g\mod N)$;
\item $b:F\to \sX$ is induced by the map $G\times E\to \sX$,
$(g,e)\mapsto g\cdot a(e)$.
\end{itemize}
There is a morphism $\mu:Y\to \sX/N$ that sends the triple
$(F\to T,E'\subset F,b:F\to \sX)$ to $(E'\to T,a=b_{|E'}:E'\to \sX)$.
We see that $\mu\circ\lambda\simeq\id_{\sX/N}$ by applying the
definitions of the objects. We can see also that
$\lambda\circ\mu\simeq\id_Y$ because if $F\to T$ is a $G$-torsor
and $E'\subset F$ is $N$-stable and an $N$-torsor over $S$,
then the morphism $G\times^NE'\to F$ induced by the inclusion
$E'\subset F$ is a morphism of $G$-torsors, hence automatically
an isomorphism. Hence $\sX/N\simeq Y$. Henceforth we write
$(\sX/N)^{\str}=Y$.

The stack $(\sX/N)^{\str}$ is endowed with a natural $G$-action
by $g\cdot (F,E',b)=(F,g(E'),b)$ where $g(E')\subset F$ is the
image of $E'$ by $g:F\to F$. As a consequence of the fact that $N$
is normal in $G$, the subspace $g(E')\subset F$ is $N$-stable,
which ensures that it is an $N$-torsor and $g\cdot (F,E',b)$
is a well-defined point of $(\sX/N)^{\str}$. We now
prove that $\mu:(\sX/N)^{\str}\to \sX/N$ is $G$-equivariant.
For this, recall that the $G$-action on $\sX/N$ is described
by $g\cdot (E,a)=({}^gE,g\circ a)$ where ${}^gE$ is the
$N$-torsor with underlying space $E$ and action defined by
$h\star e\defeq g^{-1}hge$. Now pick sections
$(F,E',b)\in (\sX/N)^{\str}(T)$ and $g\in G(T)$. From the
equivariance property of $b$ we see that $g$ provides
an isomorphism of $N$-torsors:
\[
\begin{tikzcd}[column sep=6mm,row sep=6mm]
{}^gE' \ar[rr,"g"] \ar[rd,"g\circ b"'] & & g(E') \ar[ld,"b"] \\
& \sX. &
\end{tikzcd}
\]
This shows that $\mu$ is $G$-equivariant. Of course, the
restricted action of $N$ on $(\sX/N)^{\str}$ is strictly
trivial, hence there is an action of $G/N$. We leave it to
the reader to verify that the map $(\sX/N)^{\str}\to \sX/N\to \sX/G$
induces an isomorphism $(\sX/N)^{\str}/(G/N)\isomto \sX/G$.
\end{proof}

\appendix

\section{Some results around purity}
\label{appendix:weil_restriction}

We shall make use of the results of \cite{SGA3.2}, Exp.~VIII,
\S~6 on the representability of Weil restriction of closed
subschemes and its consequences for fixed points, transporters,
normalizers, etc. It is convenient to state a version of these
results that is flexible enough to cover our needs, namely
Theorem~\ref{theo:Weil_restriction} and its corollaries. This is
best achieved using the notion of {\em pure morphism}. This
is defined in \cite{RG71} for schemes and extends without
difficulty to algebraic spaces (or even stacks but we have no
need for this); see for instance \cite{Ro11}, Appendix~B.
We recall the definition and the main facts we shall use.

Most of this material is also covered in
\cite{SGA3.1}, Exp.~VI$_{\mathrm{B}}$, \S~6 and
\cite{SGA3.2}, Exp.~XII, \S~9.

\begin{definition*} \label{def_purity}
A morphism of schemes or algebraic spaces $X\to S$ locally of
finite type is called {\em pure} if for each point $s\in S$
with henselization $(\tilde S,\tilde s)\to (S,s)$, and each point
$\tilde x\in \tilde X:=X\times_S\tilde S$ which is an associated
point in its fibre, the closure of $\tilde x$ in $\tilde X$ meets
the special fibre $X\otimes k(\tilde s)$.
\end{definition*}

For example, if $X\to S$ is proper then it is pure, because
the image in $\tilde S$ of the closure of $\tilde x$ is closed
and nonempty, hence contains $\tilde s$.
Another important example is given in \cite{RG71}, Premi\`ere
partie, Ex.~(3.3.4)(iii). For the convenience of the reader,
we provide this example with detailed explanation.

\begin{lemma*}
Let $X\to S$ be a morphism which is flat, locally of finite
presentation, with geometrically irreducible fibres without
embedded components. Then $X$ is $S$-pure.
\end{lemma*}

Note that irreducible implies nonempty by definition.

\begin{proof}
By definition, replacing $S$ by its henselization at an
arbitrary point $s$, we may assume that~$S$ is local henselian
and we have to prove that the closure of a point $x'\in X$ which is
associated in its fibre $X_{s'}$, $s'=f(x')$, meets the special
fibre~$X_s$. Let $Z$ be the closure of $s'$ in $S$. Since $S$
is local, $Z$ meets $s$ and hence we may replace~$S$ by $Z$
and assume that $S$ is irreducible with generic point $s'$.
Since $X\to S$ is open with irreducible fibres, it follows that
$X$ is irreducible, see \spref{004Z}. Now the fibre $X_{s'}$
is irreducible without embedded component, hence the assassin
$\Ass(X_{s'})$ is a single point, that is $\Ass(X_{s'})=\{x'\}$.
This means that~$x'$ is the generic point of the generic fibre,
hence the generic point of $X$. It follows
that its closure is equal to~$X$, and meets the special fibre.
\end{proof}

\begin{corollary*} \label{coro:connected_fibres_is_pure}
Let $G\to S$ be a group algebraic space which is flat, locally
of finite presentation, with connected fibres. Then $G$ is
$S$-pure.
\end{corollary*}

\begin{proof}
A connected group space over a field is a scheme and is
geometrically irreducible (\cite{SGA3.1}, Exp.~VI$_{\mathrm{A}}$, Thm.~2.6.5). Moreover any group scheme over a field is locally
complete intersection (\cite{SGA3.1}, Exp.~VII$_{\mathrm{B}}$,
Cor.~5.5.1) hence without embedded points. It follows from the
above lemma that $G\to S$ is pure.
\end{proof}

\begin{lemma*} \label{lemma:composition_of_pure}
Let $X\to Y\to S$ be morphisms locally of finite type.
Assume that $X\to Y$ is flat and pure, and $Y\to S$ is pure.
Then the composition $X\to S$ is pure.
\end{lemma*}

\begin{proof}
We may assume that $S$ is henselian with closed point $s_0$.
Let $x\in X$ be a point and $y\in Y$ its image. Assume
that $x\in\Ass(X/S)$. By flatness of $X\to S$, this means that
$x\in\Ass(X/Y)$ and $y\in\Ass(Y/S)$, see \cite{RG71} (3.2.4).
By purity of $Y\to S$, the closure of $y$ in $Y$ meets the
special fibre in a point $y_0$. Then the henselization
$(\tilde Y,\tilde y_0)\to (Y,y_0)$ hits the point $y$; let
$\tilde y\in \tilde Y$ be a lift of~$y$. Let
$\tilde x\in \tilde X\defeq X\times_Y \tilde Y$
be a lift of $x$. By invariance of the assassin under \'etale
localization (\cite{RG71} Lemma~3.4.5) and stability of purity
by base change (\cite{RG71} Corollaire~3.3.7), upon replacing
$\{Y,y,X,x\}$ by $\{\tilde Y,\tilde y,\tilde X,\tilde x\}$
we may assume that $Y$ is henselian. Since $X\to Y$ is
pure, the closure of $x$ in $X$ meets the special fibre
$X_{y_0}$. Since $X_{y_0}\subset X_{s_0}$ we are done.
\end{proof}

\begin{corollary*} \label{coro:extension_is_pure}
Let $1\to G'\to G\to G''\to 1$ be an extension of flat group
spaces of finite presentation. If $G'\to S$ and $G''\to S$ are
pure, then $G\to S$ is pure.
\end{corollary*}

\begin{proof}
The morphism $G\to G''$ is a torsor under the flat, pure,
finitely presented $G''$-group space $G'_{G''}$. By flat
descent of purity (\cite{RG71} Corollaire~3.3.7), it follows
that $G\to G''$ is pure. Then the result follows from
Lemma~\ref{lemma:composition_of_pure}.
\end{proof}

Corollaries~\ref{coro:connected_fibres_is_pure}
and~\ref{coro:extension_is_pure} show that if an $S$-group
algebraic space has a finite composition series whose factors
are reductive or proper, flat, finitely presented, then it is
pure. This includes finitely presented group schemes
of multiplicative type, because they are canonically an extension
of a finite, flat group scheme of multiplicative type by a torus.

\begin{theorem*} \label{theo:Weil_restriction}
Let $X\to S$ be a morphism of finite presentation, flat and
pure, and let $Z\to X$ be a closed immersion. Then the Weil
restriction $\Res_{X/S}Z$ is representable by a closed subscheme
of $S$. If moreover $Z\to X$ is of finite presentation, then
$\Res_{X/S}Z\to S$ also.
\end{theorem*}

\begin{proof}
For an arbitrary immersion $Z\to X$, see \cite{AR12}, Prop.~B.3.
The complement when $Z\to X$ is of finite presentation, is
standard; see for instance \cite{LMB00}, Proposition~4.18.
\end{proof}

The next two corollaries appear in \cite{SGA3.1}, Exp.~VI, \S~6.9.

\begin{corollary*} \label{coro:representability_equalizer}
Let $X,Y$ be $S$-schemes with $X\to S$ flat, pure, finitely
presented and $Y\to S$ is separated. Then the functor
$\Hom(X,Y)$ is separated over $S$: for any two $S$-morphisms
$f,g:X\to Y$ the equalizer $\Eq(f,g)\subset S$ defined by the
condition $f=g$ is representable by a closed subscheme of $S$.
If moreover the diagonal of $Y$ is of finite presentation, then
$\Eq(f,g)\into S$ also.
\end{corollary*}

\begin{proof}
Apply Theorem~\ref{theo:Weil_restriction} to the closed
immersion $(f,g)^{-1}(\Delta_Y)\into X$ where the source is the
preimage of the diagonal $\Delta_Y\into Y\times Y$
by $(f,g):X\to Y\times Y$.
\end{proof}

\begin{corollary*} \label{coro:representability_normalizer}
Let $G\to S$ be a group scheme and $H\into G$ a closed
subgroup scheme which is flat, pure, of finite presentation
over $S$. Then the functor $\Norm_G(H)$ defined as the normalizer
of $H$ in $G$ is representable by a closed subgroup scheme of $G$.
If moreover $H\into G$ is of finite presentation, then
$\Norm_G(H)\into G$ also.
\end{corollary*}

\begin{proof}
Apply Theorem~\ref{theo:Weil_restriction} to the Weil
restriction along the projection $H\times G\to G$ of the closed
immersion $c^{-1}(H)\into H\times G$ where $c:H\times G\to G$,
$(h,g)\mapsto ghg^{-1}$ is the conjugation map.
\end{proof}

\begin{corollary*} \label{coro:subgroup_locus}
Let $H\to S$ be a group scheme which is separated and of finite
presentation. Let $L\subset H$ be a proper, flat, finitely
presented closed subscheme. Then the subfunctor of $S$ defined
by the condition that $L$ is a subgroup scheme is representable
by a closed, finitely presented subscheme of $S$.
\end{corollary*}

\begin{proof}
Let $e:S\to H$ be the neutral section of $H$. Let
$a=m_{|L\times L}:L\times L\to H\times H\to H$ be the
restriction of the multiplication of $H$. Let
$b=i_{|L}:L\to H\to H$ be the restriction of the inversion.
To say that $L$ is a subgroup scheme is to say that
the closed immersions $e^{-1}(L)\into S$,
$a^{-1}(L)\into L\times L$ and $b^{-1}(L)\into L$ are
isomorphisms.
By three applications of Theorem~\ref{theo:Weil_restriction}
we see that these conditions are represented by a closed
subscheme of $S$.
\end{proof}

\begin{lemma*} \label{lemma:Hom_proper_separated}
Let $G\to S$ be a proper, flat, finitely presented group
scheme. Let $H\to S$ be a group scheme which is separated and
of finite presentation.
\begin{trivlist}
\itemn{1} The functor $\Hom(G,H)$ is representable by an
$S$-algebraic space separated and locally of finite
presentation.
\itemn{2} If moreover $G$ is finite and $H$ is affine,
then $\Hom(G,H)$ is affine and of finite presentation.
\itemn{3} The inclusion $\Mono(G,H)\subset \Hom(G,H)$ is
representable by open immersions.
\end{trivlist}
\end{lemma*}

\begin{proof}
(1) Let $T=\Hom_{\Sch}(G,H)$ be the functor of
morphisms of schemes. This is representable by an $S$-algebraic
space separated and locally of finite presentation, as follows
from e.g. \spref{0D1C}. Let $f:G_T\to H_T$ be the universal
point. Then $\Hom(G,H)$ is the subfunctor of $T$ that equalizes
the two maps $f\circ m_G$ and $m_H\circ (f\times f)$.
This is representable by a closed subscheme by
Corollary~\ref{coro:representability_equalizer}.

\smallskip

\noindent (2) This is \cite{SGA3.2},
Exp.~XI, Prop.~3.12.b) whose statement requires $G$ to be
of multiplicative type but whose proof does not use this
assumption -- and explicitly points it out.

\smallskip

\noindent (3) Let $T=\Hom(G,H)$ and let $f:G\to H$ be the
universal homomorphism over $T$. The kernel $N=\ker(f)$ is a
closed subgroup scheme of $G$ of finite presentation, and the
functor $\Mono(G,H)$ is the subfunctor of $T$ which renders
the neutral section $e:T\to N$ an isomorphism. The latter
condition is equivalent to $N\to T$ being a closed immersion;
since $N\to T$ is proper, it follows from~\spref{05XA}
that the subfunctor of interest is representable
by an open subscheme of $T$.
\end{proof}

\begin{lemma*} \label{lemma:functor_proper_subgroups}
Let $H\to S$ be a group scheme which is separated and of finite
presentation. Then the $S$-functor of subgroup schemes $L\subset H$
which are proper, flat, finitely presented is representable
by an algebraic space separated and locally of finite
presentation.
\end{lemma*}

\begin{proof}
Let $T=\Hilb$ be the Hilbert scheme of proper, flat, finitely
presented closed subschemes of~$H$. This is representable
by an algebraic space separated and locally of finite
presentation, see \spref{0D01}.
It follows from Corollary~\ref{coro:subgroup_locus} that
the functor of subgroup schemes is representable by a closed
subscheme of $T$.
\end{proof}

\bigskip

\noindent
Matthieu ROMAGNY,
{\sc Univ Rennes, CNRS, IRMAR - UMR 6625, F-35000 Rennes, France} \\
Email address: {\tt matthieu.romagny@univ-rennes1.fr}

\end{document}